\documentclass[10pt]{amsart}
\usepackage[left=1.20in, right=1.20in, top=1.4in, bottom=1.26in]{geometry}

\newcommand\inlinesum{\textstyle\sum\nolimits}

\def\publicationmode{}
\ifdefined\publicationmode
    \newcommand\suppress[1]{} 
    \newcommand\antisuppress[1]{#1}
\else
    \newcommand\suppress[1]{#1}
    \newcommand\antisuppress[1]{} 
\fi

\usepackage[shortlabels]{enumitem}
\usepackage{mathrsfs}
\usepackage{amsfonts}
\usepackage{amsmath}
\usepackage{amssymb}
\usepackage{amsthm}
\usepackage{stmaryrd}
\usepackage{textcomp}
\usepackage{mathtext}
\usepackage{yfonts}
\usepackage{tikz}
\usepackage{sidecap}
\usetikzlibrary{matrix,arrows}

\theoremstyle{plain}
\newtheorem{theorem}{Theorem}[section]
\newtheorem{lemma}[theorem]{Lemma}

\newtheorem{proposition}[theorem]{Proposition}

\newtheorem{corollary}[theorem]{Corollary}
\newtheorem*{theorem*}{Theorem}

\theoremstyle{definition}
\newtheorem{example}[theorem]{Example}
\newtheorem{remark}[theorem]{Remark}

\theoremstyle{plain}
\newtheorem*{proposition*}{Proposition}

\newcommand{\calB}{{\mathcal B}}

\newcommand{\calK}{{\mathcal K}}
\newcommand{\calL}{{\mathcal L}}

\newcommand{\calU}{{\mathcal U}}
\newcommand{\calV}{{\mathcal V}}
\newcommand{\calW}{{\mathcal W}}

\newcommand{\bbN}{{\mathbb N}}
\newcommand{\bbR}{{\mathbb R}}
\newcommand{\bbQ}{{\mathbb Q}}
\newcommand{\bbT}{{\mathbb T}}
\newcommand{\bbV}{{\mathbb V}}
\newcommand{\bbZ}{{\mathbb Z}}

\usepackage{float}
\restylefloat{table}
\restylefloat{figure}

\makeatletter
\newcommand\suchthat{\@ifstar
  {\mathrel{}\middle\vert\mathrel{}}
  {\mid}}
\makeatother

\newcommand{\card}[1]{\left|#1\right|}
\newcommand{\func}{\delta}
\newcommand{\Mult}{M}
\newcommand{\Leb}{\calL}
\DeclareMathOperator{\dist}{\Delta}
\newcommand{\Dist}{\dist}

\newcommand{\equivalent}{\Longleftrightarrow}

\usepackage{comment}

\newcommand{\Ball}{\mathcal{B}}
\newcommand{\Indices}[2]{\{#1,\dots,#2\}}
\let\originalleft\left
\let\originalright\right
\renewcommand{\left}{\mathopen{}\mathclose\bgroup\originalleft}
\renewcommand{\right}{\aftergroup\egroup\originalright}

\newcommand{\extrapoint}{x_\star}
 \setlist{nosep} 

\begin{document}

\title
[On Lipschitz partitions of unity and the Assouad--Nagata dimension]
{On Lipschitz partitions of unity\\and the Assouad--Nagata dimension}

\author{Martin W.\ Licht}
\email{martin.licht@epfl}
\address{\'Ecole Polytechnique F\'ed\'erale de Lausanne (EPFL), 1005 Lausanne, Switzerland}

\maketitle

\begin{abstract}
    We show that the standard partition of unity subordinate to an open cover of a metric space 
    has Lipschitz constant $\max(1,\Mult-1)/\Leb$, 
    where $\Leb$ is the Lebesgue number and $\Mult$ is the multiplicity of the cover. 
    If the metric space satisfies the approximate midpoint property,
    as length spaces do, 
    then the upper bound improves to $(\Mult-1)/(2\Leb)$.
    These Lipschitz estimates are optimal. 
    We also address the Lipschitz analysis of $\ell^{p}$-generalizations of the standard partition of unity, their partial sums, and their categorical products.
    Lastly, we characterize metric spaces with Assouad--Nagata dimension $n$ as exactly those metric spaces
    for which every Lebesgue cover admits an open refinement with multiplicity $n+1$ 
    while reducing the Lebesgue number by at most a constant factor.
\end{abstract}

\section{Introduction}\label{section:introduction}

Partitions of unity subordinate to an open cover are fundamental in topology, geometry, and analysis.
In numerous applications where the underlying space is a metric space,
we are interested in partitions of unity that feature Lipschitz regularity and that have the same index set as the original open cover.
The primary goal of this work is finding \emph{optimal} bounds on those Lipschitz constants in terms of only two parameters of the open cover: 
the \emph{Lebesgue number} and the \emph{multiplicity}.
The following paraphrases our main result.

\begin{theorem*}
If $\calU = ( U_\alpha )_{ \alpha \in \kappa }$ is an open cover of a metric space $(X,\dist)$
with Lebesgue number $\Leb > 0$ and multiplicity $\Mult$,
then the standard partition of unity $( \lambda_{\alpha} )_{ \alpha \in \kappa }$ subordinate to $\calU$ satisfies
\begin{displaymath}
    \forall \alpha \in \kappa : \forall x,y \in X : 
    \left| \lambda_{\alpha}(x) - \lambda_{\alpha}(y) \right| \leq \max(1,\Mult-1) \Leb^{-1} \dist(x,y)
    . 
\end{displaymath}
If, additionally, $X$ has the approximate midpoint property, then  
\begin{displaymath}
    \forall \alpha \in \kappa : \forall x,y \in X : 
    \left| \lambda_{\alpha}(x) - \lambda_{\alpha}(y) \right| \leq \frac 1 2 (\Mult-1) \Leb^{-1} \dist(x,y)
    . 
\end{displaymath}
Under the given assumptions, these bounds are the best possible for the standard partition of unity.
The same bounds apply to arbitrary partial sums of this partition of unity. 
\end{theorem*}

We now give a short overview of the literature and provide further background.
It is easily seen that any open cover of a metric space has a subordinate partition of unity whose member functions are locally Lipschitz~\cite[Theorem~5.3]{luukkainen1977elements}.
In fact, there exist partitions of unity whose members are Lipschitz; see~\cite{cobzacs2019lipschitz} and the references therein~\cite{robinson1983uniformly,fried1984open,frolik1984existence}.
However, stronger assumptions are needed if we want the partition of unity to have the same index set as the original open cover and the Lipschitz constants to be uniformly bounded. 
We address quantitative bounds that depend on the multiplicity $\Mult$ and a Lebesgue number $\Leb$ of the cover. 
To the author's best knowledge, $\Mult/\Leb$ has been the best bound for the Lipschitz constant of the standard partition of unity subordinate to the cover as reported in the literature so far~\cite{buyalo2008hyperbolic};
see also, e.g.,~\cite{bell2003property,dadarlat2007uniform,duan2010property,dydak2016large,xia2019strong}.
Here, the \emph{canonical} or \emph{standard} partition of unity is the one defined in terms of distance functions.
We show that the Lipschitz constant of that partition of unity is at most $\max(1,\Mult-1)/\Leb$.
If the metric space satisfies the approximate midpoint property~\cite{aviles2019complete}, 
which means it is a dense subset of a complete length space, 
then this further improves to $(\Mult-1)/(2\Leb)$.
Recall that the standard partition of unity is modeled on $\ell^{1}$ normalization. 
Partitions of unity modeled on more general $\ell^{p}$ normalization are of interest in settings of higher smoothness. 
We estimate their Lipschitz constants as well.

Bounding the Lipschitz constants of a partition of unity subordinate to an open cover of a metric space is of intrinsic interest and has broad utility in analysis. Furthermore, the vector of values $\Lambda$ of the partition of unity, which is the categorical product of the functions and which we henceforth call \emph{vectorization}, describes the barycentric coordinates of any point in the nerve complex associated to any pointwise finite open cover~\cite{bell2003property,dydak2003partitions,dadarlat2007uniform,dranishnikov2009cohomological,cencelj2013asymptotic,austin2014partitions}.
Using a new identity, we bound the Lipschitz constant of the vectorization with respect to the $\ell^q$-metric
by $\sqrt[q]{2} \max(1,\Mult-1) / \Leb$, or even by $\sqrt[q]{2} (\Mult-1) / (2\Leb)$ if the approximate midpoint property holds, 
which further improves the bound $\sqrt[q]{2\Mult}(\Mult+1)/\Leb$ reported previously~\cite{buyalo2008hyperbolic}.

The upper bounds for the Lipschitz constant depend on the multiplicity $\Mult$ of the cover. 
We therefore wish to address the following additional question: 
how can we modify any open cover to reduce the multiplicity while not decreasing the Lebesgue number too much? 
Specifically, we seek a decrease of the Lebesgue number by merely a constant factor that depends only on the metric space itself. 
We connect this to the following characterization of the Assouad--Nagata dimension:
a metric space has \emph{Assouad--Nagata dimension at most $n$}
if and only if there exists a constant $c > 0$ such that every open cover with some Lebesgue number $\Leb$
has an open refinement with Lebesgue number $c \Leb$ and multiplicity $n+1$.
For subsets of Euclidean space, we estimate $c \geq 1/\left( \sqrt{8n}(n+1) \right)$ using thick triangulations.
This characterization of the Assouad--Nagata dimension, albeit so far not widely discussed, 
is a very natural analogue to known characterizations of the asymptotic dimension~\cite{lang2005nagata,bell2008asymptotic,le2015assouad,dydak2016large}
and of general interest. 

The remainder of this manuscript is structured as follows.
Section~\ref{section:topology} introduces definitions and discusses partitions of unity and their vectorization within a topological setting. 
Section~\ref{section:lipschitz} introduces the metric setting and our quantitative main results. 
Lastly, Section~\ref{section:nagata} discusses aspects of the Assouad--Nagata dimension.

\section{Partitions of unity over topological spaces}\label{section:topology}

We commence our discussion within a purely topological setting.
Let $X$ be any topological space.
An \emph{open cover} of $X$ is a family $\calU = ( U_{\alpha} )_{\alpha \in \kappa}$, indexed over some set $\kappa$, of open subsets of $X$ whose union equals $X$. 
We call an open cover $\calU$ \emph{pointwise finite} if every $x \in X$ is contained in at most finitely-many members of $\calU$,
and we call $\calU$ \emph{locally finite} if every $x \in X$ has an open neighborhood that intersects at most finitely-many members of $\calU$.
We say that $\calU$ has \emph{multiplicity} $\Mult$ if each point $x \in X$ is included in at most $\Mult$ distinct members of $\calU$.

A \emph{partition of unity over $X$} is a family of non-negative continuous functions $\Lambda = ( \lambda_{\alpha} )_{\alpha \in \kappa}$, indexed over some set $\kappa$, that are pointwise almost always zero and which sum up to one everywhere, that is, 
\begin{gather*}
    \forall x \in X : \card{\{ \alpha \in \kappa \suchthat \lambda_{\alpha}(x) > 0 \}} < \infty,
    \qquad
    \forall x \in X : \inlinesum_{\alpha \in \kappa} \lambda_{\alpha}(x) = 1.
\end{gather*}
We call $\Lambda$ \emph{locally finite} if every point $x \in X$ has an open neighborhood over which only finitely many members of $\Lambda$ are not identically zero. 
We call $\Lambda$ \emph{locally vanishing} if for each $x \in X$ and $\epsilon > 0$ there exists an open neighborhood $V_x$ of $x$ such that all but finitely many members of $\Lambda$ are strictly less than $\epsilon$ over $V_x$.

We say that the partition of unity $\Lambda$ is \emph{weakly subordinate} to some cover $\calU$ 
if 
\begin{gather*}
    \forall \lambda \in \Lambda : \exists U \in \calU : 
    \{ x \in X \suchthat \lambda(x) > 0 \} \subseteq U.
\end{gather*}
We say that the partition of unity $\Lambda = ( \lambda_{\alpha} )_{\alpha \in \kappa}$ is \emph{subordinate}, or \emph{strongly subordinate}, to some cover $\calU = ( U_{\alpha} )_{\alpha \in \kappa}$ provided that 
\begin{gather*}
    \forall \alpha \in \kappa : 
    \{ x \in X \suchthat \lambda_{\alpha}(x) > 0 \} \subseteq U_{\alpha}.
\end{gather*}
Every partition of unity $( \lambda_{\alpha} )_{\alpha \in \kappa}$ is (strongly) subordinate to its \emph{induced} open cover, which is the cover given by $(\{ x \in X \suchthat \lambda_{\alpha}(x) > 0 \})_{\alpha \in \kappa}$.
We are interested in partial sums of partitions of unity. 
Given any arbitrary subset of the index set $\mu \subseteq \kappa$, we introduce the partial sum 
\begin{gather*}
    \lambda_{\mu}(x) := \inlinesum_{ \alpha \in \mu } \lambda_{\alpha}(x)
    . 
\end{gather*}
\begin{remark}
    The literature generally uses the above two notions of \emph{partition of unity subordinate to a cover} interchangeably. 
    Indeed, the distinction between weakly subordinate and (strongly) subordinate partitions of unity is a mere technicality for the purposes of topology: 
    if, say, some locally finite continuous partition of unity is weakly subordinate to an open cover, then taking partial sums yields a continuous partition of unity strongly subordinate to that open cover. 
    However, the distinction matters once quantitative properties are our concern.     
    Lastly, an even stricter notion of partition of unity subordinate to an open cover requires that the support of each function, a closed set, be included in a member of the open cover. However, this notion is not our main interest.
\end{remark}

\begin{example}
    Consider the metric Hedgehog space $X$ with $\kappa$-many spines, glued at the common zero.
    We think of $X$ as $\kappa$-many intervals $[0,1]_\alpha$, $\alpha \in \kappa$, that share the origin,
    and carrying the topology induced by the taxi cab metric.
    We introduce a continuous function $\lambda_{\alpha} : X \to [0,1]$ with support in the spine $[0,1]_\alpha$.
    Together with $\lambda_0(x) = 1 - \sum_{\alpha\in\kappa} \lambda_{\alpha}(x)$,
    we obtain a partition of unity $\Lambda$; it is generally neither locally vanishing nor even locally finite.
    If we pick $\lambda_{\alpha}(x) = x$ for every $x \in [0,1]_\alpha$,
    then $\Lambda$ is locally vanishing but not locally finite.
    Lastly, one easily finds a locally finite partition of unity; details are left to the reader.
\end{example}

While the partition of unity is a family of functions, 
we can also reinterpret it as a single vector-valued function. 
With some minor abuse of notation, we introduce the \emph{vectorization} of
$\Lambda = (\lambda_{\alpha})_{\alpha\in\kappa}$ as the function 
\begin{gather*}
    \Lambda : X \to \bbR^{\kappa}, \quad x \mapsto ( \lambda_{\alpha}(x) )_{\alpha \in \kappa}.
\end{gather*}
For every $x \in X$, the vector $\Lambda(x)$ has almost all entries equal to zero. 

The question of whether $\Lambda$ is continuous arises whenever $\bbR^{\kappa}$ carries a topology. 
We can equip $\bbR^{\kappa}$ with the product topology, which is defined as the topology of pointwise convergence of sequences.
Another topology on $\bbR^{\kappa}$ is the uniform topology: as the name suggests, a sequence converges in that topology if it converges pointwise and that pointwise convergence is uniform.
The set $\bbR^{\kappa}$ can also be equipped with the box topology:
a sequence $(x_i)_{i \in \bbN}$ converges in the box topology if and only if it converges uniformly and if, additionally, 
there is a finite subset $\mu \subseteq \kappa$ and $N \in \bbN$ such that $x_{N,\alpha} = x_{i,\alpha}$ for all $\alpha \in \kappa\setminus\mu$ and $i > N$. 
The box topology is generally finer than the uniform topology, which in turn is generally finer than the product topology;
they coincide if $\kappa$ is finite.

We characterize these topologies on $\bbR^\kappa$ via bases for their open sets,
beginning with the finest topology. 
A basis for the box topology is given by the sets of the form $\prod_{\alpha \in \kappa} (x_\alpha-\epsilon_\alpha,x_\alpha+\epsilon_\alpha)$, where $x = (x_\alpha)_{\alpha\in\kappa} \in \bbR^\kappa$ and $\epsilon = (\epsilon_\alpha)_{\alpha\in\kappa} \in (0,1)^\kappa$.
A basis for the uniform topology is given by the sets of the form $\prod_{\alpha \in \kappa} (x_\alpha-\epsilon,x_\alpha+\epsilon)$ where $x = (x_\alpha)_{\alpha\in\kappa} \in \bbR^\kappa$ and $\epsilon > 0$. 
Lastly, the product sets $\prod_{\alpha \in \kappa} I_\alpha$ where all but finitely many of the open $I_\alpha \subseteq \bbR$ are all of $\bbR$ define a basis of the open sets in the product topology.

\begin{proposition}
    Let $\Lambda=(\lambda_{\alpha})_{\alpha \in \kappa}$ be a partition of unity over a topological space $X$.
    Let $\Lambda : X \to \bbR^\kappa$ be its vectorization. 
    Then $\Lambda$ is continuous 
if $\bbR^{\kappa}$ carries the product topology,
or if $\bbR^{\kappa}$ carries the uniform topology and $\Lambda$ is locally vanishing,
or if $\bbR^{\kappa}$ carries the box topology and $\Lambda$ is locally finite.
\end{proposition}
\begin{proof}
    The function $\Lambda$ is continuous with respect to the product topology, 
    as is easily seen by the universal property of the product topology. 

    Let $U$ be an open subset of $\bbR^{\kappa}$ in the uniform topology
    and let $x \in X$ with $\Lambda(x) \in U$. Then $\Lambda(x)$ has an open neighborhood 
    $U_x = \prod_{\alpha\in\kappa} ( \lambda_{\alpha}(x)-\epsilon,\lambda_{\alpha}(x)+\epsilon)$ for some $\epsilon \in (0,1)$
    satisfying $U_x \subseteq U$.
    There exists a finite subset $\mu \subseteq \kappa$ such that $\func_\beta(x) = 0$ for all $\beta \notin \mu$. 
    If $\Lambda$ is locally vanishing, 
    then $x$ has an open neighborhood $V_x$ such that 
    \begin{gather*}
        \forall \beta \in \kappa\setminus\mu : 
        \forall y \in V_x : 
        \lambda_\beta(y) < \epsilon,
        \qquad 
        \forall \alpha \in \mu : 
        \forall y \in V_x : 
        | \lambda_{\alpha}(y) - \lambda_{\alpha}(x) | < \epsilon,
    \end{gather*}
    and therefore $\Lambda(V_x) \subseteq U_x$. 
    Since $\Lambda^{-1}(U) = \cup_{ x \in \Lambda^{-1}(U) } V_x$,
    we see that $\Lambda$ is continuous if $\Lambda$ is locally vanishing and $\bbR^{\kappa}$ carries the uniform topology.

    Let $U$ be an open subset of $\bbR^{\kappa}$ in the box topology
    and let $x \in X$ with $\Lambda(x) \in U$. Then $\Lambda(x)$ has an open neighborhood 
    $U_x = \prod_{\alpha\in\kappa} ( \lambda_{\alpha}(x)-\epsilon_\alpha,\lambda_{\alpha}(x)+\epsilon_\alpha)$ for some $\epsilon_\alpha \in (0,1)$, $\alpha \in \kappa$, satisfying $U_x \subseteq U$.
    If $\Lambda$ is locally finite, 
    then $x$ has an open neighborhood $V_x$ and there exists $\mu \subseteq \kappa$ finite such that 
    \begin{gather*}
        \forall \beta \in \kappa\setminus\mu : y \in V_x : \lambda_\beta(y) = 0,
        \qquad 
        \forall \alpha \in \mu : y \in V_x : | \lambda_{\alpha}(y) - \lambda_{\alpha}(x) | < \epsilon_{\alpha},
    \end{gather*}
    and therefore $\Lambda(V_x) \subseteq U_x$. 
By an argument analogous to the above, we see that $\Lambda$ is continuous if $\Lambda$ is locally finite and $\bbR^{\kappa}$ carries the box topology.
\end{proof}

The set $\bbR^{\kappa}$ seems to be an excessively large choice of codomain for $\Lambda$,
since the latter's values have all but a finite number of entries zero. 
We let $\bbR^{(\kappa)}$ be the subset of those members of $\bbR^\kappa$ whose entries are almost all zero. 
To round up the discussion, 
we also introduce the set $\bbR^{[\kappa]}$ of those members of $\bbR^{\kappa}$ whose entries are uniformly bounded.
Clearly, $\bbR^{(\kappa)} \subseteq \bbR^{[\kappa]} \subseteq \bbR^{\kappa}$.
These sets carry vector space structures. 
Let us review the interactions of these vector space structures with the different topologies.

\begin{proposition}
    The vector spaces $\bbR^\kappa$, $\bbR^{[\kappa]}$, and $\bbR^{(\kappa)}$ are topological vector spaces
    when carrying the product topology, the uniform topology, and the box topology, respectively. 
\end{proposition}
\begin{proof}
    Let $\bbR^\kappa$ carry the product topology. 
    The addition in $\bbR^\kappa$ is the $\kappa$-fold product of the addition of $\bbR$ and is thus continuous.
    Similarly, multiplying scalars and vectors is continuous. 
    Hence, $\bbR^{\kappa}$ with the product topology is a topological vector space. 

    Consider any neighborhood $U_z$ of $z \in \bbR^{\kappa}$ that takes the form \suppress{\begin{gather*} U_z = \textstyle\prod\nolimits_\alpha (z_\alpha - \epsilon_\alpha, z_\alpha + \epsilon_\alpha) \end{gather*}}
    \antisuppress{$U_z = \textstyle\prod\nolimits_\alpha (z_\alpha - \epsilon_\alpha, z_\alpha + \epsilon_\alpha)$}
    where $\epsilon_\alpha \in (0,1)$. 
    Let $x,y \in \bbR^{\kappa}$ with $x + y = z$ with respective neighborhoods
    \begin{gather*}
        U_x = \textstyle\prod\nolimits_\alpha (x_\alpha - \epsilon_\alpha/2, x_\alpha + \epsilon_\alpha/2),
        \quad 
        U_y = \textstyle\prod\nolimits_\alpha (y_\alpha - \epsilon_\alpha/2, y_\alpha + \epsilon_\alpha/2). 
    \end{gather*}
    By construction, $U_x + U_y \subseteq U_z$. 
    If $U$ is an open set in the box topology over $\bbR^\kappa$,
    then every $z \in U$ has got an open neighborhood $U_z \subseteq U$ of the form above. 
    Since $U_x$ and $U_y$ are open in the box topology,
    we conclude that addition in $\bbR^\kappa$ is continuous with respect to the box topology. 
    Additionally,
    if $U$ is an open set in the uniform topology, 
    then $z \in U$ has got an open neighborhood $U_z \subseteq U$ of the form above where uniformly $\epsilon_\alpha = \epsilon \in (0,1)$, and we see that $U_x$ and $U_y$ are open in the uniform topology.
    It follows that addition in $\bbR^\kappa$ is continuous with respect to the uniform topology.

    Let $w,z \in \bbR^{\kappa}$ and $\theta \in \bbR$ with $\theta w = z$.
    Let $U_z$ be a neighborhood of $z$ as above. 
    We set
    \begin{gather*} 
        U_\theta = ( \theta - \epsilon', \theta + \epsilon' ),
        \quad 
        U_w = \textstyle\prod\nolimits_\alpha ( w_\alpha - \epsilon'_\alpha, w_\alpha + \epsilon'_\alpha )
    \end{gather*}
    for some choices of $\epsilon', \epsilon'_\alpha \in (0,1)$.
    Evidently, $U_\theta U_w \subseteq U_z$ if for all $\alpha \in \kappa$ we have 
    \begin{gather*}
        \epsilon' |w_\alpha| 
        +
        \theta \epsilon'_\alpha
        +
        \epsilon' \epsilon'_\alpha
        <
        \epsilon_\alpha
        .
    \end{gather*}
    We consider two cases. 
    First, suppose that $U \subseteq \bbR^{[\kappa]}$ is an open set in the uniform topology of $\bbR^{[\kappa]}$,
    so that every $z \in U$ has a neighborhood $U_z \subseteq U$ of the form above 
    where $\epsilon_\alpha = \epsilon$ is constant.
    We pick $\epsilon' = \epsilon'_\alpha > 0$ so small that 
    \antisuppress{$\epsilon' |w_\alpha| + \theta \epsilon' + \epsilon' \epsilon' < \epsilon_\alpha$.}
    \suppress{\begin{gather*}
        \epsilon' |w_\alpha| 
        +
        \theta \epsilon'
        +
        \epsilon' \epsilon'
        <
        \epsilon_\alpha
        .
    \end{gather*}}
    Then $U_w$ is open in the uniform topology and $U_\theta U_w \subseteq U_z$.
    So scalar multiplication is continuous when $\bbR^{[\kappa]}$ carries the uniform topology.\footnote{Alternatively, we can also equip with $\bbR^{[\kappa]}$ with the supremum norm and verify that this Banach space induces the uniform topology.}

    Second, suppose that $U \subseteq \bbR^{(\kappa)}$ is an open set in the box topology of $\bbR^{(\kappa)}$,
    and let $z \in U$. 
    Then $z$ has a neighborhood $U_z \subseteq U$ of the form above,
    where $\epsilon_\alpha > 0$ is arbitrary again.
    There exists a finite $\mu \subseteq \kappa$ such that $w$ and $z$ have vanishing entries outside of $\mu$.
Choosing 
    \begin{gather*}
        \epsilon' 
        < 
        \frac 1 3 
        \cdot 
        \frac{ \textstyle\min\nolimits_{ \alpha \in \mu } \epsilon_\alpha }
        { \textstyle\max\nolimits_{\alpha\in\mu}|w_{\alpha}| },
        \quad 
        \epsilon'_\alpha < \frac{\epsilon_\alpha}{3} \cdot \frac{1}{ |\theta| + \epsilon' }
    \end{gather*}
    then ensures that $U_\theta U_w \subseteq U_z$.
    Evidently, $U_w$ is an open neighborhood of $w$ in the box topology. 
    So scalar multiplication is continuous when $\bbR^{(\kappa)}$ carries the box topology.
\end{proof}

\section{Lipschitz estimates}\label{section:lipschitz}

In what follows, we let $X$ be a metric space,
and its metric is written $\dist$.
We use the notation $\Ball(x,r)$ for the open ball of radius $r > 0$ centered at $x \in X$. 
The following definition singles out an important subclass of metric spaces:
we say that $X$ satisfies the \emph{approximate midpoint property} if 
for all $x, y \in X$ and $\epsilon > 0$ there exists $z \in X$ such that 
\begin{displaymath}
    \max( \Dist(x,z), \Dist(z,y) ) \leq \Dist(x,y)/2 + \epsilon. 
\end{displaymath}
Given $x, y \in X$, $K \in \bbN$ and $\epsilon > 0$, 
we define an \emph{$\epsilon$-approximate $K$-discrete path from $x$ to $y$} as a finite sequence 
$x_{0}, x_{1}, \dots, x_{K+1} \in X$ such that 
\begin{displaymath}
    x_{0} = x,
    \quad 
    x_{K+1} = y,
    \quad 
    \forall j \in \Indices{0}{K} : \Dist(x_{j},x_{j+1}) \leq \frac{\Dist(x,y)}{K+1} + \epsilon. \end{displaymath}
The approximate midpoint property holds if there is always an $\epsilon$-approximate $1$-discrete path from $x$ to $y$.

Let us further study that notion. 
Suppose that we are given such an $\epsilon$-approximate $K$-discrete path 
$x_{0}, x_{1}, \dots, x_{K+1}$ from $x$ to $y$.
By the triangle inequality, 
\begin{displaymath}
    \Dist(x_{i},x_{j})
    \leq
    \frac{j-i}{K+1} \Dist(x,y) + (j-i) \epsilon
    .
\end{displaymath}
Analogously, we must have 
\begin{gather*}
    \Dist(x_{i},x_{j}) \geq \frac{j-i}{K+1} \Dist(x,y) - (K+1-(j-i)) \epsilon,
\end{gather*}
as follows from 
\begin{align*}
    \Dist(x,y)
    &
    \leq 
    \Dist(x,x_{i})
    +  
    \Dist(x_{i},x_{j})
    +  
    \Dist(x_{j},y)
    \\&
    \leq 
    \Dist(x_{i},x_{j})
    +  
    \frac{K+1-(j-i)}{K+1} \Dist(x,y) + (K+1-(j-i)) \epsilon
    .
\end{align*}
In summary, the distance between $x_{i}$ and $x_{j}$ is comparable to $(j-i) \dist(x,y) / (K+1)$ up to an additive term 
that can be chosen arbitrarily small provided that $\epsilon$ is small enough.

\begin{lemma}\label{lemma:APMimpliespaths}
    Let $X$ be a metric space with approximate midpoint property.
    Then for all $x, y \in X$, all $K \in \bbN$, and $\epsilon > 0$ 
    there exists an $\epsilon$-approximate $K$-discrete path from $x$ to $y$.
\end{lemma}
\begin{proof}
    We first show that there exist $\epsilon$-approximate $K$-discrete paths when $K+1 = 2^k$ for some $k \in \bbN$.
By the approximate midpoint property, the statement holds when $K=k=1$. 
    Suppose it holds when $K+1 = 2^k$ for some number $k \geq 1$.
    Let $x, y \in X$ and $\epsilon > 0$.
    Let $\epsilon' > 0$ with $3\epsilon'/2 < \epsilon$. 
    There exist $x_0, x_1, \dots, x_{K+1} \in X$ with $x = x_0$ and $y = x_{K+1}$ and 
    \begin{displaymath}
        \dist(x_i,x_{i+1}) \leq \frac{ \dist(x,y) }{K+1} + \epsilon'.
    \end{displaymath}
    By the approximate midpoint property, there exist $z_0, z_1, \dots, z_{K} \in X$ satisfying
    \begin{displaymath}
        \max(\dist(x_i,z_{i}), \dist(z_i,x_{i+1}))
        \leq
        \dist(x_i,x_{i+1}) / 2 + \epsilon'
        \leq
        \frac{\dist(x,y)}{2K+2} + \epsilon'/2 + \epsilon'
        .
    \end{displaymath}
    We then define 
    \begin{align*}
        x'_0 &= x_{0}, \quad x'_{2} = x_{1}, \quad \dots \quad x'_{2K  } = x_{K}, \quad x'_{2K+2} = x_{K+1},
        \\
        x'_1 &= z_{0}, \quad x'_{3} = z_{1}, \quad \dots \quad x'_{2K+1} = z_{K}.
    \end{align*}
    By construction, $x'_0, x'_1, \dots, x'_{2K+2}$ is a $3\epsilon'/2$-approximate $(2K+1)$-discrete path from $x$ to $y$.
By induction, there exist $\epsilon$-approximate $K$-discrete paths if $K+1 = 2^k$ for some $k \in \bbN$.

    Now let $L \in \bbN$ and $\epsilon > 0$. 
    We let $\epsilon' > 0$ such that $\epsilon' \dist(x,y) < \epsilon/2$. 
    Upon choosing $k$ large enough and writing $K+1 = 2^{k}$, 
    we pick integers $a_0 = 0 < a_1 < \dots < a_{L} < a_{L+1} = K+1$ with \begin{displaymath}
        \frac{i}{L+1} \leq \frac{a_i}{K+1} \leq \frac{i}{L+1} + \epsilon', \quad i = 0,\dots,L+1.
    \end{displaymath}
    We let $\epsilon'' > 0$ such that $\epsilon'' (K+1) < \epsilon/2$. 
    There exists an $\epsilon''$-approximate $K$-discrete path $z_0,\dots,z_{K+1}$ from $x$ to $y$.
    We define $x_{i} := z_{a_i}$ and observe 
    \begin{align*}
        \dist( z_{a_i}, z_{a_{i+1}} ) 
        &
        \leq 
        \frac{a_{i+1}-a_{i}}{K+1} \dist(x,y) + (K+1)\epsilon''
        \\&
        \leq 
        \frac{1}{L+1} \dist(x,y) + \epsilon' \dist(x,y) + (K+1)\epsilon'' 
        \leq 
        \frac{1}{L+1} \dist(x,y) + \epsilon
        .
    \end{align*}
The sequence $z_0,\dots,z_{L+1}$ is an $\epsilon$-approximate $L$-discrete path from $x$ to $y$.
\end{proof}

\begin{remark}
    The incomplete metric space $\bbQ$ has got the approximate midpoint property.
    Complete metric spaces are length spaces if and only if they satisfy the approximate midpoint property. 
    Dense subsets of complete length spaces have the approximate midpoint property,
    and a metric space has the approximate midpoint property if a dense subspace does.
    We can therefore identify metric spaces with approximate midpoint property 
    as the class of dense subsets of complete length spaces. 
    If the metric space is even a geodesic length space, then we always pick ``$0$-approximate $K$-discrete paths'' between points. 
Effectively, $\epsilon$-approximate $K$-discrete paths serve as a stand-in for geodesics in non-geodesic metric spaces.
\end{remark}
\suppress{
    Suppose that $X$ has a dense subset $D \subseteq X$ with the approximate midpoint property.
    Let $x_n, y_n$ be sequences in $D$ that converge to $x,y \in X$.
    Let $z_n$ be a sequence of $\epsilon$-midpoints in $D$.
    We then compute, when $n$ is large enough, that 
    \begin{align*}
     \dist(x,z_n) 
     &\leq 
     \dist(x,x_n) + \dist(x_n,z_n)
     \\&\leq 
     \epsilon + \dist(x_n,z_n)
     \\&\leq 
     \epsilon + \frac 1 2 \dist(x_n,y_n) + \epsilon
     \\&\leq 
     \epsilon + \frac 1 2 \left( \dist(x,x_n) + \dist(x,y) + \dist(y_n,y) \right) + \epsilon
     \\&\leq 
     \epsilon + \frac 1 2 \left( \epsilon + \dist(x,y) + \epsilon \right) + \epsilon
     \leq 
     3 \epsilon + \frac 1 2 \dist(x,y)
     .
    \end{align*}
    Hence $X$ has the approximate midpoint property as well. 
}

Let now $\calU = (U_\alpha)_{\alpha\in\kappa}$ be an open cover of the space $X$.
A very important class of functions that are associated with that open cover are the \emph{distance functions} 
\begin{gather*}
    \func_\alpha(x) := \dist( x, X \setminus U_\alpha ), \quad \alpha \in \kappa.
\end{gather*}
That these functions have Lipschitz constant $1$ is well-known.\footnote{For example, $\func_\alpha$ is the pointwise infimum of the $1$-Lipschitz functions $d(x,y)$ as $y$ ranges over $U_\alpha$.}
Each $\func_\alpha$ is positive over $U_\alpha$. 
The function $\delta_\alpha$ equals $\infty$ in the special case $U_\alpha = X$ 
and it is finite everywhere in all other cases. 
For all $x \in X$ and $\alpha \in \kappa$ we have $\Ball(x,\func_\alpha(x)) \subseteq U_\alpha$.
The \emph{Lebesgue number of $\calU$ at $x$} is the supremum
\begin{gather*}
    \Leb_{\calU}(x) := \sup_{\alpha \in \kappa} \func_\alpha(x), \quad x \in X.
\end{gather*}
If $\Leb_{\calU}$ is infinite anywhere, then it is infinite everywhere. 
If $\Leb_{\calU}$ is finite, then it is a continuous function over $X$ and has Lipschitz constant $1$. 
If the open cover $\calU$ is pointwise finite, 
then the value $\Leb_{\calU}(x)$ is the maximum radius $r>0$ such that the open ball $\Ball(x,r)$ lies within one of the members of $\calU$.

We call $\Leb > 0$ a \emph{Lebesgue number} of an open cover $\calU$ if 
\begin{gather*}
    \forall x \in X : \exists U \in \calU : \Ball(x,\Leb) \subseteq U.
\end{gather*}
If such a Lebesgue number exists, then the \emph{optimal Lebesgue number} is defined as 
\begin{gather*}
    \Leb_\ast := \inf_{ x \in X } \Leb_{\calU}(x) = \inf_{ x \in X } \sup_{ \alpha \in \kappa } \func_\alpha(x).
\end{gather*}
If the open cover $\calU$ is pointwise finite, 
then each open ball of radius $\Leb_{\ast}$ around some $x \in X$ lies within one of the members of $\calU$.
We are interested in how the optimal Lebesgue number compares to the pointwise sum of the distance functions.
This depends on the properties of the metric space.

\begin{lemma}\label{lemma:estimateforlambdasum}
    Let $X$ be a metric space with a pointwise finite open cover $\calU$.
    Let $x \in X$ and let $\alpha_1, \alpha_2, \dots, \alpha_K \in \kappa$ be such that 
    $\func_{\alpha_1}(x) \geq \func_{\alpha_2}(x) \geq \dots \geq \func_{\alpha_K}(x) > 0$
    and $\func_{\alpha}(x) = 0$ for all $\alpha \in \kappa \setminus \{\alpha_1, \alpha_2, \dots, \alpha_K\}$.
    Then
    \begin{displaymath}
        \func_{\alpha_1}(x) \geq \Leb_\ast .
    \end{displaymath}
    If $X$ has the approximate midpoint property, then 
    \begin{displaymath}
        \func_{\alpha_1}(x) + \func_{\alpha_2}(x) \geq 2\Leb_\ast 
.
    \end{displaymath}
\end{lemma}
\begin{proof}
    If $X \in \calU$, then $\func_{\alpha_1}(x) = \infty$ and there is nothing to show. 
    So let us assume that $X \notin \calU$.
    The first inequality holds by definition. 
    Let us now assume that $X$ has the approximate midpoint property.
    Let $x \in X$. 
    Let $\alpha \in \kappa$ such that $\func_\alpha(x) \geq \func_\beta(x)$ for all $\beta \in \kappa$. 
    Then $x \in U_\alpha$.
    If there exists $\gamma \in \kappa$ with $\gamma \neq \alpha$ and $\func_\alpha(x) + \func_\gamma(x) \geq 2 \Leb_\ast$,
    then there is nothing to show. 
    So let us assume that $\func_\alpha(x) = ( 2 - \theta ) \Leb_\ast$ for some $\theta \in (0,1]$
    and that $\func_\gamma(x) \leq \eta \Leb_\ast$ for some $\eta \in [0,\theta)$ and all $\gamma \in \kappa$ with $\gamma \neq \alpha$. 
    We show how this leads to a contradiction. 
    
    For every $\epsilon > 0$ there exists $c \in X \setminus U_\alpha$ such that $\dist(x,c) < \func_\alpha(x) + \epsilon$.
    We pick an $\epsilon$-approximate $K$-discrete path from $c$ to $x$. 
    For all $k = 0, \dots, K+1$ we have 
    \begin{displaymath}
        \dist(c,x_k) \leq \frac{k}{K+1} \dist(x,c) + \epsilon,
        \quad 
        \dist(x,x_k) \leq \frac{K+1-k}{K+1} \dist(x,c) + \epsilon.
    \end{displaymath}
    In what follows, we abbreviate $\sigma_{k} := \frac{k}{K+1} \in [0,1]$. 
    Since $\func_\alpha$ is $1$-Lipschitz and $\func_\alpha(c) = 0$, we see that 
    \begin{align*}
        \func_\alpha(x_k) 
        &
        \leq 
        \frac{k}{K+1} \dist(x,c) + \epsilon
\leq 
        \sigma_{k} \func_\alpha(x) + \sigma_{k} \epsilon + \epsilon
\leq 
        \sigma_{k} ( 2 - \theta ) \Leb_\ast + 2 \epsilon
        .
    \end{align*}
    At the same time, \begin{align*}
        \func_\gamma(x_k) \leq 
        \func_\gamma(x) + \dist(x,x_k)
        &
        \leq 
        \suppress{
        \func_\gamma(x) + \frac{K+1-k}{K+1} \dist(x,c) + \epsilon
        \\&
        =
        }
        \eta \Leb_\ast + \dist(x,c) - \sigma_{k} \dist(x,c) + \epsilon
        \\&
        \leq
        \eta \Leb_\ast + \func_\alpha(x) + \epsilon - \sigma_{k} \func_\alpha(x) + \epsilon
        \suppress{
        \\&
        =
        \eta \Leb_\ast 
        + 
        ( 2 - \theta ) \Leb_\ast 
        - 
        \sigma_{k} ( 2 - \theta ) \Leb_\ast
        + 
        2\epsilon 
        \\&
        }
        \leq 
        ( 2 - \theta + \eta ) \Leb_\ast 
        - 
        \sigma_{k} ( 2 - \theta ) \Leb_\ast
        + 
        2 \epsilon
        .
    \end{align*}
    First, we have got $\func_\alpha(x_k) < \Leb_\ast$ under the condition that 
    \begin{align*}
        \sigma_{k} ( 2 - \theta ) \Leb_\ast + 2 \epsilon < \Leb_\ast
        &
        \suppress{
        \quad\equivalent\quad 
        ( 2 - \theta ) \Leb_\ast \sigma_{k} < \Leb_\ast - 2\epsilon
        \\&
        }
        \quad\equivalent\quad 
        ( 2 - \theta ) \sigma_{k} < 1 - 2\epsilon/\Leb_\ast
        .
    \end{align*}
    Second, we have got $\func_\gamma(x_k) < \Leb_\ast$ under the condition that 
    \begin{align*}
        ( 2 - \theta + \eta ) \Leb_\ast - \sigma_{k} ( 2 - \theta ) \Leb_\ast + 2 \epsilon < \Leb_\ast
        &
        \suppress{
        \quad\equivalent\quad 
        ( 1 - \theta + \eta ) \Leb_\ast + 2 \epsilon <  \sigma_{k} ( 2 - \theta ) \Leb_\ast
        \\&
        \quad\equivalent\quad 
        ( 1 - \theta + \eta ) \Leb_\ast + 2 \epsilon < \sigma_{k} ( 2 - \theta ) \Leb_\ast
        \\&
        }
        \quad\equivalent\quad 
        1 - ( \theta - \eta ) + 2 \epsilon / \Leb_\ast < \sigma_{k} ( 2 - \theta ) 
        .
    \end{align*}
    Under the condition that $\epsilon < (\theta - \eta) \Leb_\ast / 4$, 
    there exists $s \in [0,1]$ such that 
    \begin{gather*}
        1 - (\theta - \eta) + 2 \epsilon/\Leb_\ast 
        <  
        s
        <
        1 - 2\epsilon/\Leb_\ast 
        .
    \end{gather*}
    Choose $K \in \bbN$ and $0 \leq k \leq K+1$ such that $\sigma_{k}$ is sufficiently close to $s$. 
    Then $\func_\beta(x_k) < \Leb_\ast$ for all $\beta \in \kappa$. 
    But that contradicts the definition of the optimal Lebesgue number. 
    Therefore, $\func_{\alpha_1}(x) + \func_{\alpha_2}(x) \geq 2\Leb_\ast$
    if $X$ has the approximate midpoint property.
\end{proof}

We are now in a position to address our main result:
we construct a Lipschitz partition of unity 
whose Lipschitz properties are controlled by the multiplicity and Lebesgue number of the original cover. 
We will also measure the Lipschitz continuity of the vectorization of any partition of unity.
The specific partitions of unity that we analyze include what is known as the canonical partition of unity in the literature.

\begin{theorem}\label{theorem:mainresult}
    Let $\calU = ( U_\alpha )_{\alpha \in \kappa}$ be a pointwise finite open cover of $X$ with Lebesgue number $\Leb > 0$ and multiplicity $\Mult \in \bbN$.
Let $p \in [1,\infty)$.
    Then the functions
    \begin{gather*}
        \lambda_{\alpha}(x) = \frac{ \func_\alpha(x)^{p} }{ \sum_{\beta\in\kappa} \func_\beta(x)^{p} },
        \quad 
        \alpha \in \kappa,
    \end{gather*}
    constitute a continuous partition of unity $\Lambda = ( \lambda_{\alpha} )_{\alpha \in \kappa}$ subordinate to $\calU$.
    We have 
    \begin{gather*}
        \forall \mu \subseteq \kappa :
        \forall x,y \in X : 
        \left| \lambda_{\mu}(x) - \lambda_{\mu}(y) \right|
        \leq
        p \frac{\max(1,\Mult-1)^{\frac 1 p}}{\Leb} \dist(x,y)
        .
    \end{gather*}
    If $X$ has the approximate midpoint property, then 
    \begin{gather*}
        \forall \mu \subseteq \kappa :
        \forall x,y \in X : 
        | \lambda_{\mu}(x) - \lambda_{\mu}(y) | 
        \leq 
        p 2^{1 - \frac 1 p} \frac{(\Mult-1)^{\frac 1 p}}{2\Leb} \dist(x,y)
        ,
        \\
        \forall \mu \subseteq \kappa :
        \forall x,y \in X : 
        | \lambda_{\mu}(x) - \lambda_{\mu}(y) | 
        \leq 
        p C_p \frac{(\Mult-1)}{2\Leb} \dist(x,y)
        ,
    \end{gather*}
    where $C_{p} \in [1,2)$ is a constant that only depends on $p$ and is described in the proof. 
\end{theorem}
\begin{proof} 
If $X \in \calU$, then $\Lambda$ includes constant functions for each copy of $X$ in $\calU$, and all other functions are zero.
    Let us assume from here on that $X \notin \calU$. 
    
    If $\Mult=1$, then $\calU$ consists of open pairwise disjoint sets with pairwise distance at least $\Leb > 0$.
    The first inequality is evident then,
    and if $X \notin \calU$ and $\Mult = 1$, then $X$ does not satisfy the approximate midpoint property.

Let $p \in [1,\infty)$. For each $\alpha \in \kappa$, we set 
    \begin{gather*}
        \func_\alpha(x) 
        := 
        \dist( x, X \setminus U_\alpha ),
        \quad 
        \func_{\alpha,p}(x) = \func_\alpha(x)^p
        .
    \end{gather*}
    Each $\func_\alpha$ is a non-negative function positive over $U_\alpha$ and has Lipschitz constant $1$.
    Next, we introduce 
    \begin{gather*}
        \lambda_{\alpha,p}(x) 
        := 
        \func_\alpha(x)^{p} 
        \cdot 
        \left( \inlinesum_{ \beta \in \kappa } \func_\beta(x)^{p} \right)^{-1}
        .
    \end{gather*}
    The function $\lambda_{\alpha,p}$ is well-defined as $\calU$ is a pointwise finite cover of $X$.
    Given any $\mu \subseteq \kappa$, we let  
    \begin{gather*}
        \func_{\mu,p} := \inlinesum_{ \beta \in \mu } \func_{\beta,p},
        \quad 
        \lambda_{\mu,p} := \inlinesum_{ \beta \in \mu } \lambda_{\beta,p}.
    \end{gather*}
    Again, these are well-defined functions because $\calU$ is pointwise finite. 
    By construction, $\lambda_{\mu,p}$ takes values within $[0,1]$ and $\lambda_{\kappa,p} = 1$ everywhere.
    In particular, $\Lambda := ( \lambda_{\alpha,p} )_{\alpha\in\kappa}$ is the partition of unity subordinate to $\calU$ 
    that we study.

We proceed with some general observations on the partial sums of the distance functions and some first estimates.
    In what follows, the sets of indices 
    \begin{gather*}
        S(\mu,x,y) 
        := 
        \{\alpha\in\mu\suchthat\func_\alpha(x)>\func_\alpha(y)\}
        ,
        \quad 
        T(\mu,x,y) 
        := 
        \{\alpha\in\kappa\setminus\mu\suchthat\func_\alpha(x)>\func_\alpha(y)\}
    \end{gather*}
    and their possible sizes will be important.
    Consider any $\mu \subseteq \kappa$.
    For all $\alpha \in \mu$ and all $x, y \in X$, 
    we either have $\func_\alpha(x) \leq \func_\alpha(y)$, so that
    $\func_{\alpha}(x)^p \leq \func_{\alpha}(y)^p$, 
    or we have $\func_\alpha(x) > \func_\alpha(y)$ and so $\func_\alpha(x) > 0$,
    in which case
    \begin{align*}
        \func_{\alpha}(x)^p
        =
        \func_{\alpha}(y)^p + p \int_{\func_{\alpha}(y)}^{\func_{\alpha}(x)} g^{p-1} dg 
        &
        \leq
        \func_{\alpha}(y)^p + p \func_{\alpha}(x)^{p-1} \left( \func_{\alpha}(x)- \func_{\alpha}(y)\right)
        \\&
        \leq
        \func_{\alpha}(y)^p + p \func_{\alpha}(x)^{p-1} \dist(x,y)
        .
    \end{align*}
    Summing over $\alpha \in \mu$, we thus conclude that  
    \begin{gather*}
        \func_{\mu,p}(x)
        -
        \func_{\mu,p}(y)
        \leq
        \sum_{ \alpha \in S(\mu,x,y) }
        \func_{\alpha}(x)^{p-1} 
        \cdot p\dist(x,y)
        = 
        \func_{S(\mu,x,y),p-1}(x)
        \cdot p\dist(x,y)
        .
    \end{gather*}
    \suppress{\color{blue}Analogously, 
    \begin{align*}
        \func_{\mu,p}(y)
        &
        \leq
        \func_{\mu,p}(x)
        +
        \sum_{ \alpha \in S(\mu,y,x) }
        \func_{\alpha}(y)^{p-1} 
        \cdot p\dist(x,y)
        .
    \end{align*}\color{black}}
    Arguing analogously to bound $\func_{\mu,p}(y) - \func_{\mu,p}(x)$ from above, 
    we obtain an upper bound 
    \begin{gather*}
        |\func_{\mu,p}(x) - \func_{\mu,p}(y)|
        \leq 
        \max
        \left( 
            \func_{S(\mu,y,x),p-1}(y), \func_{S(\mu,x,y),p-1}(x)
        \right)
        \cdot p\dist(x,y)
        .
    \end{gather*}
    Note that the sums above are only ever taken over up to $\Mult$ indices.

We now study the Lipschitz constants of the functions that constitute the partition of unity and their partial sums.
    Consider any $\mu \subseteq \kappa$ and $x, y \in X$.
    We have   
    \begin{align*}
        \lambda_{\mu,p}(x) - \lambda_{\mu,p}(y)
        &=
        \frac{ \func_{\mu,p}(x) }{ \func_{\kappa,p}(x) }
        -
        \frac{ \func_{\mu,p}(y) }{ \func_{\kappa,p}(y) }
=
        \frac{ \func_{\mu,p}(x) \func_{\kappa,p}(y) - \func_{\mu,p}(y) \func_{\kappa,p}(x) }{ \func_{\kappa,p}(x) \func_{\kappa,p}(y) }
        .
    \end{align*}
    Easy algebraic manipulations show 
    \begin{align*}
\func_{\mu,p}(x) \func_{\kappa,p}(y) - \func_{\mu,p}(y) \func_{\kappa,p}(x)
        &
        =
        \func_{\mu,p}(x) \func_{\kappa\setminus\mu,p}(y) 
        + 
        \func_{\mu,p}(x) \func_{\mu,p}(y) 
        - 
        \func_{\mu,p}(x) \func_{\mu,p}(y) 
        - 
        \func_{\mu,p}(y) \func_{\kappa\setminus\mu,p}(x) 
        \\& 
        =
        \func_{\mu,p}(x) \func_{\kappa\setminus\mu,p}(y) 
        - 
        \func_{\mu,p}(y) \func_{\kappa\setminus\mu,p}(x) 
        \\& 
        = 
        \left( \func_{\mu,p}(x) - \func_{\mu,p}(y) \right) \func_{\kappa\setminus\mu,p}(y)
        -
        \func_{\mu,p}(y) \left( \func_{\kappa\setminus\mu,p}(x) - \func_{\kappa\setminus\mu,p}(y) \right) 
        .
    \end{align*}
In other words,
    \begin{align*}
        \lambda_{\mu}(x) - \lambda_{\mu}(y)
        &=
        \frac{ 
            \left( \func_{\mu,p}(x) - \func_{\mu,p}(y) \right) \func_{\kappa\setminus\mu,p}(y)
            -
            \func_{\mu,p}(y) \left( \func_{\kappa\setminus\mu,p}(x) - \func_{\kappa\setminus\mu,p}(y) \right) 
        }{ 
            \func_{\kappa,p}(x) \func_{\kappa,p}(y) 
        }
        .
    \end{align*}
    Based on that argument, 
    and a similar argument with the roles of $x$ and $y$ interchanged, 
    we derive the following upper bounds: 
    \begin{align*}
        \left| \lambda_{\mu}(x) - \lambda_{\mu}(y) \right| 
        &
        \leq
        \frac{ 
            \big| \func_{\mu,p}(x) - \func_{\mu,p}(y) \big|
            \func_{\kappa\setminus\mu,p}(y) 
            +
            \big| \func_{\kappa\setminus\mu,p}(x) - \func_{\kappa\setminus\mu,p}(y) \big|
            \func_{\mu,p}(y) 
        }{
            \func_{\kappa,p}(x) \func_{\kappa,p}(y) 
        }
        ,
        \\
        \left| \lambda_{\mu}(x) - \lambda_{\mu}(y) \right| 
        &
        \leq
        \frac{ 
            \big| \func_{\mu,p}(x) - \func_{\mu,p}(y) \big|
            \func_{\kappa\setminus\mu,p}(x) 
            +
            \big| \func_{\kappa\setminus\mu,p}(x) - \func_{\kappa\setminus\mu,p}(y) \big|
            \func_{\mu,p}(x) 
        }{
            \func_{\kappa,p}(x) \func_{\kappa,p}(y) 
        }
        ,
        \\
        \left| \lambda_{\mu}(x) - \lambda_{\mu}(y) \right| 
        & 
        \leq
        \max\left( \func_{\kappa,p}(x), \func_{\kappa,p}(y) \right)^{-1}
        \max\left( 
            \big| \func_{\mu,p}(x) - \func_{\mu,p}(y) \big|
            ,
            \big| \func_{\kappa\setminus\mu,p}(x) - \func_{\kappa\setminus\mu,p}(y) \big|
        \right)
        .
    \end{align*}
    In what follows, we write $c_\ast := 2$ if we assume that $X$ has the approximate midpoint property and $c_\ast := 1$ otherwise.
    If $x, y \in X$ with $\dist(x,y) \geq c_\ast\Leb$, then 
    \begin{align*}
        \forall \mu \subseteq \kappa 
        : 
        | \lambda_{\mu,p}(x) - \lambda_{\mu,p}(y) |
        \leq 
        1 
        \leq 
        \left( c_\ast\Leb \right)^{-1}
        \dist(x,y)
        .
    \end{align*}
    It remains to consider the case $\dist(x,y) < c_\ast\Leb$. 
    Then there exists $\beta \in \kappa$ such that $x, y \in U_\beta$.

    If there are $\Mult$ different $\alpha \in \mu$ with $\func_\alpha(x) > 0$ 
    and $\Mult$ different $\alpha \in \mu$ with $\func_\alpha(y) > 0$, 
    then we already must have $\lambda_{\mu,p}(x) = \lambda_{\mu,p}(y) = 1$. 
    Likewise, if there are $\Mult$ different $\alpha \in \kappa\setminus\mu$ with $\func_\alpha(x) > 0$ 
    and $\Mult$ different $\alpha \in \kappa\setminus\mu$ with $\func_\alpha(y) > 0$, 
    then we already must have $\lambda_{\mu,p}(x) = \lambda_{\mu,p}(y) = 0$. 
    Consequently, it only remains to analyze the situation 
    where at least one of the sets 
    $S(\mu,x,y)$ and $S(\mu,y,x)$ has size at most $\Mult-1$
    and 
    where at least one of the sets 
    $T(\mu,x,y)$ and $T(\mu,y,x)$ has size at most $\Mult-1$.
    
    If $\card{S(\mu,x,y)} = \Mult$, then $S(\mu,y,x) = T(\mu,x,y) = \emptyset$ and $\card{T(\mu,y,x)} \leq \Mult-1$,
    and with the observation $\func_{\kappa,p}(x) = \func_{\mu,p}(x)$ it then follows that 
    \begin{align*}
        \left| \lambda_{\mu}(x) - \lambda_{\mu}(y) \right| 
        &
        \leq
        \func_{\kappa,p}(y)^{-1}
        \big| \func_{\kappa\setminus\mu,p}(x) - \func_{\kappa\setminus\mu,p}(y) \big|
        \leq 
        \frac{\func_{T(\mu,y,x),{p-1}}(y)}{\func_{\kappa,p}(y)}
        \cdot p\dist(x,y)
        .
    \end{align*}
    Similarly, 
    if $\card{S(\mu,y,x)} = \Mult$, then $S(\mu,x,y) = T(\mu,y,x)=\emptyset$ and $\card{T(\mu,x,y)} \leq \Mult-1$, as well as
    \begin{align*}
        \left| \lambda_{\mu}(x) - \lambda_{\mu}(y) \right| 
        &
        \leq
        \func_{\kappa,p}(x)^{-1}
        \big| \func_{\kappa\setminus\mu,p}(y) - \func_{\kappa\setminus\mu,p}(x) \big|
        \leq 
        \frac{\func_{T(\mu,x,y),{p-1}}(x)}{\func_{\kappa,p}(x)}
        \cdot p\dist(x,y)
        .
    \end{align*}
    We handle the cases $\card{T(\mu,x,y)} = \Mult$ and $\card{T(\mu,y,x)} = \Mult$ completely analogously,
    finding either of
    \begin{align*}
        \left| \lambda_{\mu}(x) - \lambda_{\mu}(y) \right| 
        &
        \leq
        \frac{\func_{S(\mu,y,x),{p-1}}(y)}{\func_{\kappa,p}(y)}
        \cdot p\dist(x,y)
        ,
        \qquad 
        \left| \lambda_{\mu}(x) - \lambda_{\mu}(y) \right| 
        \leq 
        \frac{\func_{S(\mu,x,y),{p-1}}(x)}{\func_{\kappa,p}(x)}
        \cdot p\dist(x,y)
        ,
    \end{align*}
    where the index sets $S(\mu,y,x)$ or $S(\mu,x,y)$, respectively, have size at most $\Mult-1$.
    
Having settled these special cases, it remains to proceed from the inequality 
\begin{align*}
        \left| \lambda_{\mu}(x) - \lambda_{\mu}(y) \right| 
        \leq
        \max\left( 
            \frac{ \func_{S(\mu,x,y),p-1}(x) }{ \func_{\kappa,p}(x) }
            ,
            \frac{ \func_{S(\mu,y,x),p-1}(y) }{ \func_{\kappa,p}(y) } 
            ,
            \frac{ \func_{T(\mu,x,y),p-1}(x) }{ \func_{\kappa,p}(x) }
            ,
            \frac{ \func_{T(\mu,y,x),p-1}(y) }{ \func_{\kappa,p}(y) } 
        \right)
        \cdot 
        p \dist(x,y)
        ,
    \end{align*}
    where none of the index sets has size more than $\Mult-1$.
    We continue the proof by estimating the various fractions
    as they appear on the right-hand side of the preceding inequality.

One pathway forward is via comparisons of generalized means:
    \begin{align*}
        \frac{
            \func_{S(\mu,x,y),p-1}(x) 
        }{
            \func_{\kappa,p}(x)
        }
        \leq 
        \card{S(\mu,x,y)}^{\frac 1 p}
        \frac{ 
            \func_{S(\mu,x,y),p}(x)^{\frac{p-1}{p}}
        }{
            \func_{\kappa,p}(x)
        }
        \leq 
        \card{S(\mu,x,y)}^{\frac 1 p}
        \func_{\kappa,p}(x)^{-\frac 1 p}
        .
    \end{align*}
    Handling the other terms in the maximum similarly, 
    we arrive at a first estimate, where we can now assume all the set sizes bounded by $\Mult-1$: 
    \begin{align*}
        \left| \lambda_{\mu,p}(x) - \lambda_{\mu,p}(y) \right| 
        \leq 
        p 
        \max\left(
            \frac{ \card{S(\mu,x,y)} }{ \func_{\kappa,p}(x) }
            , 
            \frac{ \card{S(\mu,y,x)} }{ \func_{\kappa,p}(y) }
            , 
            \frac{ \card{T(\mu,x,y)} }{ \func_{\kappa,p}(x) }
            , 
            \frac{ \card{T(\mu,y,x)} }{ \func_{\kappa,p}(y) }
        \right)^{\frac 1 p}
        \dist(x,y)
        .
    \end{align*}    
    Concerning a lower bound for $\func_{\kappa,p}(z)^{\frac 1 p}$ as $z$ ranges over $X$,
    we already know that the largest value in the family $(\func_{\alpha}(z))_{\alpha\in\kappa}$ is at least $\Leb_{\ast}$, so $\func_{\kappa,p}(z)^{\frac 1 p} \geq \Leb_{\ast}$.
    If $X$ has the approximate midpoint property, then we even know that the largest two values in that family add up to at least $2\Leb_{\ast}$,
    and then another means inequality shows  
    $\func_{\kappa,p}(z)^{\frac 1 p} \geq 2^{\frac 1 p - 1} \cdot 2\Leb_{\ast}$.
    This completes the discussion of the first estimate.

    Another pathway is based on some alternative estimates:
    \begin{gather*}
        \frac{
            \func_{S(\mu,x,y),p-1}(x) 
        }{
            \func_{\kappa,p}(x)
        }
        \leq 
        \card{S(\mu,x,y)} 
        \max_{\alpha\in\kappa} 
        \frac{\func_{\alpha}(x)^{p-1}}{\func_{\kappa,p}(x)}
        .
\end{gather*}
Once again handling the other terms in the maximum similarly, 
    this leads to a second estimate: 
    \begin{align*}
        &
        \left| \func_{\mu,p}(x) - \func_{\mu,p}(y) \right| 
\leq 
        p
        \max\left(
            \card{S(\mu,x,y)} 
            , 
            \card{S(\mu,y,x)} 
            , 
            \card{T(\mu,x,y)} 
            , 
            \card{T(\mu,y,x)} 
        \right)
        \max_{ \substack{ \alpha\in\kappa \\ z \in \{x,y\} } }
        \frac{\func_{\alpha}(z)^{p-1}}{\func_{\kappa,p}(z)}
        \dist(x,y)
        .
    \end{align*}
    As in the first estimate, the set sizes in the second estimate are now assumed bounded by $\Mult-1$. 
    It remains to find upper bounds for a few remaining quantities. 
With regard to the second estimate, 
    we generally know that $\func_{\kappa,p} \geq \Leb_{\ast}^p$ over $X$, 
    implying 
    \begin{gather*}
        \max_{\alpha\in\kappa} 
        \max_{ z \in \{x,y\} } 
        \func_{\alpha}(z)^{p-1}\func_{\kappa,p}(z)^{-1} 
        \leq 
        \Leb_\ast^{-1}
        .
    \end{gather*}
    Let us from here on assume that $X$ satisfies the approximate midpoint property. 
    We then even know $\func_{\alpha} + \func_{\beta} \geq 2\Leb_\ast$, 
    where $\alpha, \beta \in \kappa$ such that $\func_\alpha(z) \geq \func_\beta(z) \geq \func_\gamma(z)$ for all other $\gamma \in \kappa \setminus \{\alpha,\beta\}$.
    We observe 
    \begin{gather*}
        \frac{
            \func_{\alpha}(z)^{p-1}
        }{
            \func_{\kappa,p}(z)
        }
        \leq 
        \frac{
            \func_{\alpha}(z)^{p-1}
        }{
            \func_{\alpha}(z)^{p} + \func_{\beta}(z)^{p}
        }
        \leq 
        \frac{
            2^{p-1}
            \func_{\alpha}(z)^{p-1}
        }{
            ( \func_{\alpha}(z) + \func_{\beta}(z) )^{p}
        }
        \leq 
        \frac{2^{p-1}}{2\Leb_\ast}
        .
    \end{gather*}
    If $\func_\alpha(z) \geq 2 \Leb_{\ast}$ or $p=1$, then 
    \begin{gather*}
        \frac{
            \func_{\alpha}(z)^{p-1}
        }{
            \func_{\kappa,p}(z)
        }
        \leq 
        \frac{
            \func_{\alpha}(z)^{p-1}
        }{
            \func_{\alpha,p}(z)
        }
        \leq 
        \frac{
            1
        }{
            \func_{\alpha}(z)
        }
        \leq 
        \frac{ 1 }{ 2\Leb_\ast }
        .
    \end{gather*}
    Otherwise, $\func_\alpha(z) = 2 \Leb_{\ast} \eta$ and $\func_\alpha(z) = 2 \Leb_{\ast} (1-\eta)$ for some $\eta \in [0.5,1]$, 
    and hence 
    \begin{gather*}
        \frac{
            \func_{\alpha}(z)^{p-1}
        }{
            \func_{\kappa,p}(z)
        }
        \leq 
        \frac{
            \func_{\alpha}(z)^{p-1}
        }{
            \func_{\alpha}(z)^{p} + \func_{\beta}(z)^{p}
        }
        \leq 
        \frac{1}{2\Leb_\ast}
        \cdot 
        \max_{\eta\in[0.5,1]}
        \frac{
            \eta^{p-1}
        }{
            \eta^{p} + ( 1 - \eta )^{p}
        }
        .
    \end{gather*}
    Assume that $p > 1$.
    The derivative of the quotient is a positive multiple of 
    \begin{gather*}
        (p-1) \eta^{p-2} ( \eta^{p} + ( 1 - \eta )^{p} )
        -
        p \eta^{p-1} ( \eta^{p-1} - ( 1 - \eta )^{p-1} )
        ,
    \end{gather*}
    \suppress{{\color{blue}
    The derivative of the quotient has the form 
    \begin{gather*}
        \frac{
            (p-1) \eta^{p-2} ( \eta^{p} + ( 1 - \eta )^{p} )
            -
            p \eta^{p-1} ( \eta^{p-1} - ( 1 - \eta )^{p-1} )
        }{
            ( \eta^{p} + ( 1 - \eta )^{p} )^{2}
        }
        ,
    \end{gather*}
    \begin{gather*}
        (p-1) \eta^{p-2} ( \eta^{p} + ( 1 - \eta )^{p} )
        =
        p \eta^{p-1} ( \eta^{p-1} - ( 1 - \eta )^{p-1} )
        \\
        (p-1) ( \eta^{p} + ( 1 - \eta )^{p} )
        =
        p \eta ( \eta^{p-1} - ( 1 - \eta )^{p-1} )
        \\
        p \eta^{p} + p ( 1 - \eta )^{p} - \eta^{p} - ( 1 - \eta )^{p} 
        =
        p \eta^{p} - p \eta ( 1 - \eta )^{p-1} 
        \\
        p ( 1 - \eta )^{p} - \eta^{p} - ( 1 - \eta )^{p} 
        =
        - p \eta ( 1 - \eta )^{p-1} 
        \\
        p ( 1 - \eta )^{p} - \eta^{p} - ( 1 - \eta )^{p-1} + \eta ( 1 - \eta )^{p-1}  
        =
        - p \eta ( 1 - \eta )^{p-1} 
        \\
        p ( 1 - \eta )^{p-1} - p \eta ( 1 - \eta )^{p-1} - \eta^{p} - ( 1 - \eta )^{p-1} + \eta ( 1 - \eta )^{p-1}  
        =
        - p \eta ( 1 - \eta )^{p-1} 
        \\
        p ( 1 - \eta )^{p-1} - p \eta ( 1 - \eta )^{p-1} - ( 1 - \eta )^{p-1} + \eta ( 1 - \eta )^{p-1} + p \eta ( 1 - \eta )^{p-1}  
        =
        \eta^{p}
        \\
        ( 1 - \eta )^{p-1}
        \left( \eta + p - 1 \right)
            =
        \eta^{p}
        .
    \end{gather*}}}which is positive at $0.5$ and negative at $1$. 
    So there exists a critical point $\eta_{\ast} \in (0.5,1)$ of the quotient,
    and one easily sees that any critical point is characterized by 
    \begin{align*}
        \eta_\ast^{p} = ( 1 - \eta_\ast )^{p-1} \left( \eta_\ast + p - 1 \right)
        &\quad\equivalent\quad 
        \left( \frac{ \eta_\ast }{ 1 - \eta_\ast } \right)^{p-1} = 1 + \frac{ p - 1 }{ \eta_\ast }
        \quad\equivalent\quad 
        \frac{ \eta_\ast }{ 1 - \eta_\ast } = \sqrt[p-1]{1 + \frac{ p - 1 }{ \eta_\ast } }
        .
\end{align*}
    The last equation characterizes $\eta_\ast$ as the intersection of two functions over $(0,1)$,
    one of which strictly increasing and the other one strictly decreasing,
    that must only occur within $(0.5,1)$. 
    So there exists a unique maximizer $\eta_\ast \in (0.5,1)$. 
    We will not compute $\eta_\ast$ but use lower bounds instead. First,    
    \begin{gather*}
        \frac{ \eta_\ast }{ 1 - \eta_\ast } \geq \sqrt[p-1]{p}
\quad\equivalent\quad 
        \eta_\ast \geq \frac{ \sqrt[p-1]{p} }{ 1 + \sqrt[p-1]{p} }
        .
    \end{gather*}
    Upon substituting $\tau_\ast = \eta_\ast (1-\eta_\ast)^{-1}$, 
    which means $\eta_\ast = \tau_\ast (1+\tau_\ast)^{-1}$, 
    the characterization reads 
    \begin{gather*}
        \tau_\ast^{p-1} = 1 + \frac{ p - 1 }{ \tau_\ast } ( 1 + \tau_\ast )
\quad\equivalent\quad 
        \tau_\ast^{p} = p \tau_\ast + (p-1)
        .
    \end{gather*}
    The latter leads us to an alternative lower bound 
    \begin{gather*}
        \tau_{\ast} \geq \sqrt[p]{2p-1}
        \quad\equivalent\quad 
        \eta_\ast \geq \frac{ \sqrt[p]{2p-1} }{ 1+\sqrt[p]{2p-1} }
        .
    \end{gather*}
    Irrespective of the value of $\eta_\ast$, we must have 
    \begin{align*}
        \max_{\eta\in[0.5,1]}
        \frac{
            \eta^{p-1}
        }{
            \eta^{p} + ( 1 - \eta )^{p}
        }
        &
        \leq  
        \frac{
            \eta_\ast^{p-1}
        }{
            \eta_\ast^{p} + \left( \eta_\ast + p - 1 \right)^{-1} ( 1 - \eta_\ast ) \eta_\ast^{p}
        }
= 
        \frac{p-1}{p} \eta_\ast^{-1} + \frac{1}{p}
        .
    \end{align*}
    \suppress{{\color{blue}
    \begin{align*}
        &
        \color{blue}
        \leq  
        \frac{
            \eta_\ast^{-1}
        }{
            1 + \left( \eta_\ast + p - 1 \right)^{-1} ( 1 - \eta_\ast ) 
        }
        \\&
        \color{blue}
        \leq  
        \frac{
            \eta_\ast^{-1} \left( \eta_\ast + p - 1 \right)
        }{
             \eta_\ast + p - 1 + 1 - \eta_\ast 
        }
        = \eta_\ast^{-1} \frac{p-1+\eta_\ast}{p}
        = \frac{p-1+\eta_\ast}{\eta_\ast p}
        = \frac{1}{\eta_\ast} - \frac{1-\eta_\ast}{\eta_\ast p}
        \\&
        \color{blue}
        = \frac{p-1+\eta_\ast}{\eta_\ast p}
        .
    \end{align*}
    }}The last expression is decreasing in $\eta_\ast$.
With the two lower bounds on $\eta_\ast$, 
the last remaining inequality is now evident, with
    \begin{gather*}
        C_p 
        := 
        \max_{\eta\in[0.5,1]}
        \frac{
            \eta^{p-1}
        }{
            \eta^{p} + ( 1 - \eta )^{p}
        }
        \leq 
        \left( 1 - \frac 1 p \right) 
        \min\left( 
\frac{ 1+\sqrt[p]{2p-1} }{ \sqrt[p]{2p-1} }
            ,
            \frac{ 1 + \sqrt[p-1]{p} }{ \sqrt[p-1]{p} }
        \right)
        + 
        \frac{1}{p}
        <
        2
        .
    \end{gather*}
    This completes our Lipschitz estimates for the partition of unity functions and their partial sums.
\end{proof}

\begin{remark}    
    The case $p=1$ arguably suffices for the purposes of Lipschitz analysis. 
    In practice, the role of the parameter $p$ is to control the ``smoothness'' of the partition of unity.
    For example, when the construction is carried out with an interval cover of the real line, 
    then $p > 1$ will force the partition of unity to have continuous derivatives.
    As a rule of thumb, increasing $p$ smoothes the partition of unity functions along the support boundaries but increases the internal gradients.
\end{remark}

\begin{remark}
    The explicit upper bound for the constant $C_{p}$ in the proof of Theorem~\ref{theorem:mainresult} can be improved when $p > 1$. 
    Its limit at $p=1$ equals one, which is optimal. 
    Solving the equation $\eta^{p} = ( 1 - \eta )^{p-1} \left( \eta + p - 1 \right)$ would provide a better estimate. 
    We do not pursue that topic any further. 
\end{remark}

Let us attend to the vectorization and its Lipschitz properties. 
Algebraically, the vectorization is a mapping $\Lambda : X \to \bbR^{(\kappa)}$, 
where $\kappa$ always denotes the index set of the partition of unity.
To study its Lipschitz properties, we first need to agree on a metric over its target space. 
For example, irrespective of $\kappa$'s cardinality,
we can equip $\bbR^{(\kappa)}$ with various well-known norms 
such as the Lebesgue norms 
\begin{align*}
    \| z \|_{\ell^\infty} 
    &:= 
    \| z \|_{\max} 
    := 
    \max_{ \alpha \in \kappa } |z_\alpha|,
    \quad 
    z \in \bbR^{(\kappa)},
    \\
    \| z \|_{\ell^q} 
    &:= 
    \left( \inlinesum_{\alpha\in\kappa} |z_\alpha|^q \right)^{\frac 1 q},
    \quad 
    z \in \bbR^{(\kappa)},
    \quad 
    q \in [1,\infty).
\end{align*}
Given $q \in [1,\infty]$, we always let $\bbR^{(\kappa)}_{\ell^{q}}$ denote the space $\bbR^{(\kappa)}$ equipped with the $\ell^{q}$ norm.\footnote{In what follows, we use the convention that $1/q = 0$ and $\sqrt[q]{a} = 1$ for any $a > 0$ in the case $q = \infty$.}

\begin{proposition}
    Let $\Lambda = ( \lambda_{\alpha} )_{\alpha \in \kappa}$ be a partition of unity over the metric space $X$ and subordinate to an open cover $\calU = ( U_{\alpha} )_{\alpha \in \kappa}$. 
    Suppose that for some constants $L, L_{\infty} \geq 0$ such that 
    \begin{gather*}
        \forall \mu \subseteq \kappa : 
        \forall x,y \in X : 
        \left| \lambda_{\mu}(x) - \lambda_{\mu}(y) \right|
        \leq
        L \dist(x,y)
        ,
        \\
        \forall \alpha \in \kappa : 
        \forall x,y \in X : 
        \left| \lambda_{\alpha}(x) - \lambda_{\alpha}(y) \right|
        \leq
        L_{\infty} \dist(x,y)
        .
    \end{gather*}
    If $q \in [1,\infty]$, then 
    \begin{gather*}
        \forall x,y \in X : 
        \left\| \Lambda(x) - \Lambda(y) \right\|_{\ell^q}
        \leq
        \sqrt[q]{2} 
        L^{1/q} L_{\infty}^{1-1/q} 
        \dist(x,y)
        .
    \end{gather*}
\end{proposition}
\begin{proof}
    Let $x, y \in X$.
Without loss of generality, $\Lambda(x) \neq \Lambda(y)$. 
    For further discussion, we define 
    \begin{align*}
        \mu := \left\{ \alpha \in \kappa \suchthat \lambda_{\alpha}(x) > \lambda_{\alpha}(y) \right\}.
    \end{align*}
    It then follows that  
    \begin{align*}
        \| \Lambda(x) - \Lambda(y) \|_{\ell^1}
        &= 
        \sum_{ \alpha \in \mu} ( \lambda_{\alpha}(x) - \lambda_{\alpha}(y) )
        -
        \sum_{ \alpha \in \kappa\setminus\mu} ( \lambda_{\alpha}(x) - \lambda_{\alpha}(y) )
        \\&= 
        \lambda_{\mu}(x) - \lambda_{\mu}(y) - \lambda_{\kappa\setminus\mu}(x) + \lambda_{\kappa\setminus\mu}(y)
        \\&= 
        \lambda_{\mu}(x) - \lambda_{\mu}(y) - ( 1 - \lambda_{\mu}(x) ) + (  1 - \lambda_{\mu}(y) )
        =
        2\lambda_{\mu}(x) - 2\lambda_{\mu}(y)
        .
    \end{align*}
    Consequently,
    \begin{align*}
        \| \Lambda(x) - \Lambda(y) \|_{\ell^1} \leq 2L \dist(x,y)
        ,
        \quad 
        \| \Lambda(x) - \Lambda(y) \|_{\ell^\infty} \leq L_{\infty} \dist(x,y)
        .
    \end{align*}
    Due to the generalized H\"older inequality, also known as Littlewood's inequality~\cite{garling2007inequalities},  \begin{align*}
        \| \Lambda(x) - \Lambda(y) \|_{\ell^q}
        &\leq 
        \| \Lambda(x) - \Lambda(y) \|_{\ell^1}^{1/q}
        \cdot
        \| \Lambda(x) - \Lambda(y) \|_{\ell^\infty}^{1-1/q}
        \suppress{
        \\&\leq 
        \sqrt[q]{2} L^{1/q} \dist(x,y)^{1/q}
        \cdot
        L_{\infty}^{1-1/q} \dist(x,y)^{1-1/q}
        }
        .
    \end{align*}
    This completes the estimates for the vectorization.
\end{proof}

\begin{corollary}[{compare~\cite[Proposition~1]{bell2003property}}]\label{corollary:vectorizationestimate}
    For every $\Mult \geq 1$ and $q \in [1,\infty]$,
    there exists $C(\Mult,q) \leq \sqrt[q]{2} \max(1,\Mult-1)$
    such that every open cover $\calU = (U_\alpha)_{\alpha\in\kappa}$ of a metric space $X$
    with Lebesgue number $\Leb$ and multiplicity $\Mult$
    admits a partition of unity whose vectorization 
    $\Lambda : X \rightarrow \bbR^{(\kappa)}_{\ell^{q}}$
    has Lipschitz constant $C(\Mult,q)/\Leb$.
    If $X$ has the approximate midpoint property,
    then $C(\Mult,q) \leq \sqrt[q]{2} (\Mult-1) / 2$. 
\end{corollary}

Let us now slightly change the perspective. 
    If we define a partition of unity of the form as Theorem~\ref{theorem:mainresult} with some parameter $p \in [1,\infty)$, 
    then the collection of their $p$-th roots at any point defines vector on the Lebesgue $p$-norm unit sphere. 
    In the limit as $p$ goes to infinity, these roots become
    \begin{align*}
\lim_{ p \rightarrow \infty }
        \left( 
            \frac{ \func_\alpha(x)^{p} }{ \inlinesum_{\beta\in\kappa} \func_\beta(x)^{p} }
        \right)^{\frac 1 p}
        =
        \frac{ \func_\alpha(x) }{ \max_{\beta\in\kappa} \func_\beta(x) },
        \quad 
        \alpha \in \kappa.
    \end{align*}
    The latter defines a family of functions, indexed over $\alpha \in \kappa$, that have maximum $1$ everywhere. 
    They have Lipschitz constant at most $2/\Leb_{\ast}$. 
    Inspired by that, we show that the $p$-th roots of the partition of unity functions are Lipschitz too.
    \suppress{
    \[
        f(x)/g(x) - f(y)/g(y)
        =
        \frac 1 {g(x)} \left( f(x) - f(y) \right)
        +
        f(y) \left( \frac 1 {g(x)} - \frac 1 {g(y)} \right)
        =
        \frac 1 {g(x)} \left( f(x) - f(y) \right)
        +
        f(y) \frac { g(x) - g(y) } {g(x)g(y)} 
        \leq 
        \frac 1 {L} \left( f(x) - f(y) \right)
        +
        \frac { g(x) - g(y) } {L} 
        \leq 
        2 / L
        .
    \]
    }

\begin{proposition}\label{proposition:roots} Let $\calU = ( U_\alpha )_{\alpha \in \kappa}$ be a pointwise finite open cover of $X$ with Lebesgue number $\Leb > 0$ and multiplicity $\Mult \in \bbN$.
Let $p \in [1,\infty)$. Then the functions
    \begin{align*}
        \lambda_{\alpha,p}(x) = \frac{ \func_\alpha(x)^{p} }{ \sum_{\beta\in\kappa} \func_\beta(x)^{p} },
        \quad 
        \alpha \in \kappa,
    \end{align*}
    constitute a continuous partition of unity $\Lambda = ( \lambda_{\alpha,p} )_{\alpha \in \kappa}$ subordinate to $\calU$.
    We have 
    \begin{gather*}
        \forall \mu \subseteq \kappa :
        \forall x,y \in X : 
        \left| \lambda_{\mu,p}(x)^{\frac 1 p} - \lambda_{\mu,p}(y)^{\frac 1 p} \right| 
        \leq 
        2^{\frac{p-1}{p}}
        \frac{(2\Mult-1)^{\frac 1 p}}{\Leb} \dist(x,y)
        .
    \end{gather*}
    If the approximate midpoint property holds, then 
    \begin{gather*}
        \forall \mu \subseteq \kappa :
        \forall x,y \in X : 
        \left| \lambda_{\mu,p}(x)^{\frac 1 p} - \lambda_{\mu,p}(y)^{\frac 1 p} \right| 
        \leq 
        2^{\frac{p-1}{p}}
        \frac{(2\Mult-1)^{\frac 1 p}}{2^{\frac 1 p}\Leb} \dist(x,y)
        .
    \end{gather*}
\end{proposition}
\begin{proof}
    We begin with some algebraic reformulation:
\begin{align*}
        \lambda_{\mu,p}(x)^{\frac 1 p} - \lambda_{\mu,p}(y)^{\frac 1 p}
        &
        =
        \frac{\sqrt[p]{\func_{\mu,p}(x)}}{\sqrt[p]{\func_{\kappa,p}(x)}}
        -
        \frac{\sqrt[p]{\func_{\mu,p}(y)}}{\sqrt[p]{\func_{\kappa,p}(y)}}
=
        \frac{
            \sqrt[p]{ \func_{\mu,p}(x) \func_{\kappa,p}(y) } 
            -
            \sqrt[p]{ \func_{\mu,p}(y) \func_{\kappa,p}(x) }
        }{
            \sqrt[p]{\func_{\kappa,p}(x)}
            \sqrt[p]{\func_{\kappa,p}(y)}
        }
        .
    \end{align*}
    We assume without loss of generality that $\lambda_{\mu,p}(x) \geq \lambda_{\mu,p}(y)$,
    since otherwise we can simply switch the role of $x$ and $y$. 
    The derivative of the $p$-th root is non-increasing, and thus 
    \begin{align*}
        &
        \sqrt[p]{ \func_{\mu,p}(x) \func_{\kappa,p}(y) } 
        -
        \sqrt[p]{ \func_{\mu,p}(y) \func_{\kappa,p}(x) } 
\leq 
        \sqrt[p]{ \func_{\mu,p}(x) \func_{\kappa\setminus\mu}(y) } 
        -
        \sqrt[p]{ \func_{\mu,p}(y) \func_{\kappa\setminus\mu}(x) } 
        \\&
        \qquad
		= 
        \sqrt[p]{ \func_{\mu,p}(x) \func_{\kappa\setminus\mu}(y) } 
        -
        \sqrt[p]{ \func_{\mu,p}(y) \func_{\kappa\setminus\mu}(y) } 
        +
        \sqrt[p]{ \func_{\mu,p}(y) \func_{\kappa\setminus\mu}(y) } 
        -
        \sqrt[p]{ \func_{\mu,p}(y) \func_{\kappa\setminus\mu}(x) } 
        .
    \end{align*}
    Therefore, 
    \begin{align*}
        &
        \lambda_{\mu,p}(x)^{\frac 1 p} - \lambda_{\mu,p}(y)^{\frac 1 p}
\leq
        \frac{
            \Big| \sqrt[p]{ \func_{\mu,p}(x) } - \sqrt[p]{ \func_{\mu,p}(y) } \Big| 
            \sqrt[p]{ \func_{\kappa\setminus\mu}(y) } 
            + 
            \Big| \sqrt[p]{ \func_{\kappa\setminus\mu}(x) } - \sqrt[p]{ \func_{\kappa\setminus\mu}(y) } \Big|
            \sqrt[p]{ \func_{\mu,p}(y) } 
        }{
            \sqrt[p]{\func_{\kappa,p}(x)}
            \sqrt[p]{\func_{\kappa,p}(y)}
        }
        .
    \end{align*}
	Next, 
	$\sqrt[p]{ \func_{\mu,p}(y) } + \sqrt[p]{ \func_{\kappa\setminus\mu,p}(y) } \leq 2^{\frac{p-1}{p}} \sqrt[p]{\func_{\kappa,p}(y)}$ by a generalized means inequality.
Given any $\mu \subseteq \kappa$, the reverse triangle inequality and another generalized means inequality lead to 
    \begin{align*}
        \left| 
            \func_{\mu,p}(x)^{\frac 1 p} - \func_{\mu,p}(y)^{\frac 1 p} 
        \right|
        &
        \leq 
        \left( \sum_{\alpha\in\mu} \left| \func_{\alpha}(x) - \func_{\alpha}(y) \right|^{p} \right)^{\frac 1 p}
\leq 
        \card{\{ \alpha \in \mu \suchthat \func_{\alpha}(x) + \func_{\alpha}(y) > 0 \}}^{\frac 1 p}
        \cdot 
        \dist(x,y)
        .
    \end{align*}
    Similarly,  
    \begin{align*}
\left| 
            \func_{\kappa\setminus\mu,p}(x)^{\frac 1 p} - \func_{\kappa\setminus\mu,p}(y)^{\frac 1 p} 
        \right|
        &
        \leq 
        \card{\{ \alpha \in \kappa\setminus\mu \suchthat \func_{\alpha}(x) + \func_{\alpha}(y) > 0 \}}^{\frac 1 p}
        \cdot 
        \dist(x,y)
        .
    \end{align*}
    The two cardinalities used above are at most $2\Mult$. 
    The first set has size $2\Mult$ only when $\lambda_{\mu,p}(x) = \lambda_{\mu,p}(y) = 1$.
    The second set has size $2\Mult$ only when $\lambda_{\mu,p}(x) = \lambda_{\mu,p}(y) = 0$.
Lastly, $\func_{\kappa,p}(x) \geq \Leb^{p}$, and if the approximate midpoint property holds,
    then $\func_{\kappa,p}(x)^{\frac 1 p} \geq 2^{-\frac{p-1}{p}} \cdot 2\Leb$. 
    The desired result thus follows.
\end{proof}

\begin{example}
    Constructing the partition of unity via the distance functions of the cover is very common in the literature and known as the \emph{canonical partition of unity}~\cite{brodskiy2005coarse}.
    The following example shows that our Lipschitz estimate for that partition of unity is optimal for general metric spaces. 
    We let $X = [-\Leb,\Leb] \cup \{\extrapoint\}$ be the disjoint union of a closed interval with the canonical metric and a one-point metric space $\{\extrapoint\}$,
    and we define $\dist(x,x_{\ast}) = \Leb$ for all $x \in [-\Leb,\Leb]$.
The open sets 
    \begin{align*}
U_1 = [-\Leb,\Leb],
        \quad 
U_2 = \dots = U_\Mult = (0,\Leb] \cup \{\extrapoint\},
    \end{align*}
    form an open cover $\calU$ of $X$. 
    By construction, $\calU$ has optimal Lebesgue number $\Leb$ and multiplicity $\Mult$.
We construct the standard partition of unity subordinate to this cover.
    To estimate the Lipschitz constant, we only need to study its behavior over $[0,\Leb]$,
	where 
    \antisuppress{$\lambda_1(0) = 1$ and $\lambda_2(0) = \dots = \lambda_M(0) = 0$ as well as }
    \begin{gather*}
        \suppress{\lambda_1(0) = 1,
        \quad
        \lambda_2(0) = \dots = \lambda_M(0) = 0, 
        \\}
        \lambda_1(\epsilon) = \frac{\Leb}{\Leb + (\Mult-1)\epsilon},
        \quad 
        \lambda_2(\epsilon) = \dots = \lambda_M(\epsilon) 
        = 
        \frac{\epsilon}{\Leb + (\Mult-1)\epsilon}
    \end{gather*}
    for any $\epsilon \in (0,\Leb]$.
    Asymptotically, as $\epsilon$ goes to zero, the quotient on the right-hand side of 
    \begin{align*}
        \frac{ \lambda_1(0) - \lambda_1(\epsilon) }{ \epsilon }
= 
        \frac{(\Mult-1)}{\Leb + (\Mult-1)\epsilon} 
\end{align*}
grows towards the proposed upper bound $(\Mult-1)/\Leb$.
    Note that $X$ does not satisfy the approximate midpoint property. 
Concerning the vectorization, given $q \in [1,\infty)$, we calculate 
    \begin{align*}
        \| \Lambda(0) - \Lambda(\epsilon) \|_q
        &
        =
( \Leb + (\Mult-1)\epsilon)^{-1}
        ( \Mult-1 ) 
        \sqrt[q]{ 1 + (\Mult-1)^{1-q} }
        \cdot
        \epsilon
        .
    \end{align*}
In the case $q=1$ and in the limit $q \to \infty$, we realize the proposed upper bound as $\epsilon$ vanishes. 
    In between, for $1 < q < \infty$ and $\Mult \geq 3$, we overestimate by a factor of at most $\sqrt[q]{2}$.
\end{example}

\begin{example}
    Let us now discuss a simple example where the metric space $X$ enjoys the approximate midpoint property. 
    Let $X = \bbR$ be the real line and consider the open cover given by
    \begin{align*}
        U_1 = (-\infty,2\Leb), 
        \quad 
        U_2 = \dots = U_\Mult = (0,\infty).
    \end{align*}
    By construction, the optimal Lebesgue number is $\Leb$ and the multiplicity is $\Mult$.
    Now, 
    \antisuppress{$\lambda_{1}(0) = 1$ and $\lambda_2(0) = \dots = \lambda_M(0) = 0$ as well as }
    \begin{gather*}
        \suppress{\lambda_1(0) = 1
        , 
        \quad 
        \lambda_2(0) = \dots = \lambda_M(0) = 0, 
        ,
        \\}
        \lambda_1(\epsilon) = \frac{2\Leb-\epsilon}{2\Leb-\epsilon + (\Mult-1)\epsilon}
        ,
        \quad 
        \lambda_2(\epsilon) = \dots = \lambda_M(\epsilon) 
        = 
        \frac{\epsilon}{2\Leb-\epsilon + (\Mult-1)\epsilon}
\end{gather*}
    for $\epsilon > 0$ small.
    Asymptotically, as $\epsilon$ approaches zero, the quotient on the right-hand side of 
    \begin{align*}
        \lambda_1(0) - \lambda_1(\epsilon) 
= 
        \frac{(\Mult-1)}{2\Leb-\epsilon + (\Mult-1)\epsilon}\epsilon
\end{align*}
    grows towards the proposed upper bound $(\Mult-1)/2\Leb$. 
    Moreover, given $q \in [1,\infty)$, one calculates
    \begin{align*}
        \| \Lambda(0) - \Lambda(\epsilon) \|_q
        &
        =
( 2\Leb-\epsilon + (\Mult-1)\epsilon)^{-1}
        ( \Mult-1 ) 
        \sqrt[q]{ 1 + (\Mult-1)^{1-q} }
        \epsilon
        ,
    \end{align*}
    and the value when $q = \infty$ is again obtained from a limit. 
    These observations are analogous to the previous example, 
    except that the estimate is lower due to the approximate midpoint property.
\end{example}

\begin{remark} 
    Let us now address some prospective improvements of these results. 
    It seems likely that Theorem~\ref{theorem:mainresult} and Proposition~\ref{proposition:roots} can be improved when $p > 1$.
More conceptually, 
    higher multiplicity of the initial cover should actually \emph{decrease} upper bounds for Lipschitz constants of subordinate partitions of unity
    because redundancy allows for greater flexibility in choosing the partition functions. 
    Obviously, any function with Lipschitz constant $L \geq 0$ is the sum of $K$ identical functions with Lipschitz constant $L/K$. 
    Duplicating members of the open cover thus allows reducing the magnitude and Lipschitz constants of any partition of unity functions.
    If the original open cover already has higher multiplicity, then there should exist a subordinate partition of unity, not necessarily the standard one, with a lower Lipschitz constant.
\end{remark}

\section{Applications within spaces of finite Assouad--Nagata dimension}\label{section:nagata}

Our estimates benefit from low multiplicities of the covers. 
For that reason, it is of interest how to find a cover with reduced multiplicity without decreasing the Lebesgue number too much. We connect that question with elements of metric dimension theory.

We introduce a few more notions related to covers of metric spaces. 
Suppose that $X$ is a metric space and that $\calU$ is an open cover. 
The \emph{Lebesgue number} and the \emph{multiplicity} of $\calU$ have already been defined earlier. 
Given $\rho > 0$, we say that $\calU$ has \emph{$\rho$-multiplicity} $\Mult$ 
if every set of diameter at most $\rho$ intersects at most $\Mult$ members of $\calU$. 
Moreover, given $\rho > 0$,
we say that a family $\calB$ of subsets of $X$ is \emph{$\rho$-disjoint}
if pairwise distinct members of $\calB$ have distance strictly more than $\rho$.
Note that a family ${\mathcal B}$ has $\rho $-multiplicity $1$ if ${\mathcal B}$ is $\rho $-disjoint.

\suppress{\color{red}
Indeed, if $B, B' \in \calB$ with $\dist(B,B') > \rho$, meaning that $\forall (x,x') \in B \times B'$ we have $\dist(x,x') > \rho$,
and if $A \subseteq X$ has diameter at most $\rho$, meaning that $\dist(x,x') \leq \rho$ for all $(x,x') \in A \times A$,
then any $x \in A \cap B$ and $x' \in A \cap B'$ must satisfy $\dist(x,x') \leq \rho$ and $\dist(x,x') > \rho$,
which is a contradiction.
\color{black}}

The following auxiliary result may be interpreted as a quantitative version of a result in Engelking's book~\cite[Lemma~5.1.8]{engelking1989general}. 

\begin{lemma}\label{lemma:ballmultiplicitygeneration}
    Let $\calU = (U_\alpha)_{\alpha\in\kappa}$ be a pointwise finite open cover of the metric space $X$ with Lebesgue number $\Leb$.
    Then $\calU$ has a locally finite open refinement $\calV = (V_\alpha)_{\alpha\in\kappa}$ with Lebesgue number $\Leb/3$
    that satisfies
    \begin{gather*}
        \forall \alpha \in \kappa : \forall x \in \overline{V_\alpha} : \overline{\Ball(x,\Leb/3)} \subseteq U_\alpha,
        \qquad
        \forall \beta \in \kappa : \forall x \in X \setminus U_\beta : \Ball(x, \Leb/3 ) \cap V_\beta = \emptyset.
    \end{gather*}
    In particular, if $\calU$ has multiplicity $\Mult$, then $\calV$ has $\Leb/3$-multiplicity $\Mult$.
\end{lemma}
\begin{proof}
    We define a cover $\calV := (V_\alpha)_{\alpha\in\kappa}$ 
    by setting $V_\alpha := \left\{ x \in X \suchthat \func_\alpha(x) > 2\Leb/3 \right\}$
    for each $\alpha \in \kappa$. 
    We know that $\Leb \leq \max_{\alpha\in\kappa} \func_\alpha(x)$ for all $x \in X$,
    and conclude that $\calV$ is an open cover of $X$ with $V_{\alpha} \subseteq U_{\alpha}$. 
    Since each $\func_\alpha$ has Lipschitz constant $1$, we see 
    \begin{align*}
        \forall x \in X : \exists \alpha \in \kappa : \Ball(x,\Leb/3) \subseteq V_\alpha
        ,
        \qquad 
        \forall \alpha \in \kappa : \forall x \in \overline{V_\alpha} : \overline{\Ball(x,\Leb/3)} \subseteq U_\alpha
        .
    \end{align*}
    So $\calV$ has Lebesgue number $\Leb/3$. 
    For each $x \in X$ there exists a minimal non-empty finite set $\mu(x) \subseteq \kappa$ with $\func_\beta(x) = 0$ for all $\beta \in \kappa\setminus\mu(x)$. 
    Again using that each $\func_\alpha$ has Lipschitz constant $1$, 
    it follows that 
    \begin{align*}
        \forall \beta \in \kappa\setminus\mu(x) : 
        \forall y \in \Ball(x,\Leb/3) : 
        \func_\beta(y) < \Leb/3
        .
    \end{align*}
    In particular, $\Ball(x,\Leb/3) \cap V_\beta = \emptyset$ for all $\beta \in \kappa\setminus\mu(x)$. 
    If $A \subseteq X$ with $x \in A$ has diameter at most $\Leb/3$,
    then $A \subseteq \overline{\Ball(x,\Leb/3)}$ can only intersect those $V_\alpha \in \calV$ 
    for which $\alpha \in \mu(x)$. 
    The proof is complete. 
\end{proof}

Our main topic in this section is how to find subcovers with reduced multiplicity while not deteriorating the Lebesgue number too much.
Specifically, we would like to see the multiplicity reduced while the Lebesgue number deteriorates only by a constant factor that depends only on the metric space itself. 
As it turns out, this naturally connects with the notion of Assouad--Nagata dimension~\cite{nagata1960metric,assouad1982sur,lang2005nagata,le2015assouad}. 

We say that a metric space $X$ has \emph{Assouad--Nagata dimension} at most $n$ if 
there is a $c > 0$ such that for all $\ell > 0$ there exists an open cover $\calU$ of $X$ whose members have diameter at most $c \ell$ and with Lebesgue number $\ell$ and multiplicity $n + 1$.
The following proposition lists numerous equivalent conditions that characterize finite Assouad--Nagata dimension, 
some of which seem not widely circulated; see also the subsequent remark.

\begin{proposition}\label{proposition:nagata}
    Let $X$ be a metric space.
    The following are equivalent.
    \begin{enumerate}[label=(\Alph*),itemsep=0pt]
        \item \label{proposition:nagata:a}
        There exists $c_a \in (0,1)$ such that every open cover $\calU = (U_{\alpha})_{\alpha\in\kappa}$ of $X$ with Lebesgue number $L$ has an open refinement $\calW = (W_{\alpha})_{\alpha\in\kappa}$ with Lebesgue number $c_a L$ and multiplicity $\Mult$.
        \item \label{proposition:nagata:b}
        There exists ${c}_b \in (0,1)$ such that every open cover $\calU = (U_{\alpha})_{\alpha\in\kappa}$ of $X$ with Lebesgue number $L$ has an open refinement $\calW = (W_{\alpha})_{\alpha\in\kappa}$ with ${c}_b L$-multiplicity $\Mult$.
        \item \label{proposition:nagata:c}
        There exists $c_c, c_c' \in (0,1)$ such that every open cover $\calU = (U_{\alpha})_{\alpha\in\kappa}$ of $X$ with Lebesgue number $L$ has an open refinement $\calW = (W_{\alpha})_{\alpha\in\kappa}$ with Lebesgue number $c_c L$ and $c_c' L$-multiplicity $\Mult$.

        \item \label{proposition:nagata:d}
        There exists ${c}_d \in (0,1)$ such that for every $\ell > 0$, there exists an $\ell$-bounded open cover of $X$ with Lebesgue number ${c}_d \ell$ and multiplicity $\Mult$.
        \item \label{proposition:nagata:e} There exists ${c}_e \in (0,1)$ such that for every $\ell > 0$, there exists an $\ell$-bounded open cover of $X$ with ${c}_e \ell$-multiplicity $\Mult$.
        \item \label{proposition:nagata:f}
        There exists ${c}_f, {c}_f' \in (0,1)$ such that for every $\ell > 0$, there exists an $\ell$-bounded open cover of $X$ with Lebesgue number ${c}_f \ell$ and ${c}_f' \ell$-multiplicity $\Mult$.

        \item \label{proposition:nagata:g}
        There exists $c_g \in (0,1)$ such that for every $\ell > 0$, 
        there exist pairwise disjoint $\ell$-bounded $c_g \ell$-disjoint open families $\calB_1, \dots, \calB_\Mult$ 
        such that $\calB = \calB_1 \cup \dots \cup \calB_\Mult$ is a cover of $X$. 
        \item \label{proposition:nagata:h}
        There exists $c_h, c_h' \in (0,1)$ such that for every $\ell > 0$, 
        there exist pairwise disjoint $\ell$-bounded $c_h' \ell$-disjoint open families $\calB_1, \dots, \calB_\Mult$ 
        such that $\calB = \calB_1 \cup \dots \cup \calB_\Mult$ is a cover of $X$ with Lebesgue number $c_h \ell$ and $c_h' \ell$-multiplicity $\Mult$. 

        \item \label{proposition:nagata:i}
        There exist constants $c_i, c_i' \in (0,1)$
        such that for every open cover $\calU$ of $X$ with Lebesgue number $\Leb$,  
        there exist pairwise disjoint $\Leb$-bounded $c_i' \Leb$-disjoint open families $\calB_{1}, \dots, \calB_{\Mult}$ 
        such that $\calB = \calB_{1} \cup \dots \cup \calB_{\Mult}$ is a cover of $X$ refining $\calU$ and with Lebesgue number $c_i \Leb$.

    \end{enumerate}
    If any of the conditions is true, then we may assume 
    \begin{gather*}
        c_{a} \geq c_{c} = c_{c}' \geq { c_{a} }/{3}, 
        \quad 
        {c}_{b} \geq c_{c} = c_{c}' \geq {c}_{b/{6}}, 
        \quad 
        {c}_{d} \geq {c}_{f} = {c}_{f}' \geq {c}_{d}/{3}, 
        \quad 
        {c}_{e} \geq {c}_{f} = {c}_{f}' \geq {{c}_{e}}/{3},
        \\
        c_{g} \geq c_{h} = c_{h}' \geq {c_{g}}/{3}, 
        \quad 
        {c}_{f} \geq c_{h}, 
        \quad 
        {c}_{f}' \geq c_{h}', 
        \quad 
        c_{g} \geq C_{0}^{-1} ( M+1 )^{-1} (M+2)^{-2} {c}_{f},
        \\
        c_{a} \geq {{c}_{d}}/{2}, 
        \quad 
        {c}_{d} \geq {c_{a}}/{2}, 
        \quad 
        c_{i} \geq {c_{h}}/{2}, 
        \quad 
        c_{h} \geq {c_{i}}/{2}, 
        \quad 
        c_{i}' \geq {c_{h}'}/{ 2}, 
        \quad 
        {c_{h}'} \geq { c_{i}'}/{ 2}.
    \end{gather*}
    Here, $C_{0} \leq \max(1,\Mult-1)$, and if $X$ has the approximate midpoint property, $C_{0} \leq (\Mult-1)/2$.
\end{proposition}
\begin{proof}
    We prove a sequence of implications.
    \begin{itemize}
        \item\ref{proposition:nagata:c}$\implies$\ref{proposition:nagata:a}:
        Clear, with $c_a \geq c_c$.
        
        \item\ref{proposition:nagata:a}$\implies$\ref{proposition:nagata:c}:
        This follows by Lemma~\ref{lemma:ballmultiplicitygeneration} with $c_c = c_c' \geq c_a/3$.

        \item\ref{proposition:nagata:c}$\implies$\ref{proposition:nagata:b}:
        Clear, with ${c}_b \geq c_c'$.

        \item\ref{proposition:nagata:b}$\implies$\ref{proposition:nagata:c}:
        Let $\calU = (U_{\alpha})_{\alpha\in\kappa}$ be an open cover of $X$ with Lebesgue number $L$. 
Whenever, $\alpha \in \kappa$, 
        we write $X_{\alpha} \subseteq X$ for the set of those $x \in X$ for which $\Ball(x,L) \subseteq U_{\alpha}$,
        and we define $V_\alpha := \left\{ z \in X \suchthat* z \in \Ball(X_{\alpha},L/2) \right\}$. 
        Then $\calV := ( V_\alpha )_{\alpha \in \kappa}$ is an open refinement of $\calU$ with Lebesgue number $L/2$.
        By assumption, $\calV$ has an open refinement $\calW = (W_{\alpha})_{\alpha\in\kappa}$ with ${c}_b L / 2$-multiplicity $\Mult$.
        Then 
        \begin{displaymath}
            \calW' := ( \Ball(W_{\alpha}, {c}_b L / 6 ) )_{\alpha\in\kappa}
        \end{displaymath}
        is an open cover with Lebesgue number ${c}_b L / 6$.
        If $\calW'$ does not have $({c}_b L / 6)$-multiplicity $\Mult$, 
        then there exists $D \subseteq X$ with a diameter of at most $c_b L /6$ 
        intersecting pairwise distinct $B_1,\dots,B_M,B_{M+1} \in \calW'$. 
        But then $\Ball( D, {c}_b L / 6 )$ has diameter at most $c_b L / 2$ and intersects $\Mult+1$ distinct members of $\calW$,
        which is a contradiction. 
        Hence $\calW'$ has $({c}_b L/6)$-multiplicity $\Mult$. 
        By construction, $\calW'$ refines $\calU$ 
        because $\Ball(W_{\alpha}, {c}_b L / 6 ) \subseteq \Ball(V_{\alpha}, {c}_b L / 6 ) \subseteq \Ball(X_{\alpha},2L/3 ) \subseteq U_{\alpha}$ for each $\alpha \in \kappa$.

        \item\ref{proposition:nagata:f}$\implies$\ref{proposition:nagata:d}:
        Clear, with ${c}_d \geq {c}_f$.

        \item\ref{proposition:nagata:d}$\implies$\ref{proposition:nagata:f}:
        This follows by Lemma~\ref{lemma:ballmultiplicitygeneration} with ${c}_f = {c}_f' \geq {c}_d/3$.
        
        \item\ref{proposition:nagata:f}$\implies$\ref{proposition:nagata:e}:
        Clear, with ${c}_e \geq {c}_f'$.

        \item\ref{proposition:nagata:e}$\implies$\ref{proposition:nagata:f}:
        Let $\calU$ be an $\ell$-bounded open cover of $X$ with ${c}_e \ell$-multiplicity $\Mult$. 
We define 
        \begin{displaymath}
            \calU' := \left\{ \Ball(U, {c}_e \ell / 3 ) \suchthat U \in \calU \right\}
            .
        \end{displaymath}
        Obviously, $\calU'$ is an open cover of $X$ with Lebesgue number ${c}_e \ell / 3$. 
        If $\calU'$ does not have $({c}_e \ell / 3)$-multiplicity $\Mult$, 
        then there exists $D \subseteq X$ with a diameter of at most $c_e \ell /3$ 
        intersecting pairwise distinct $B_1,\dots,B_M,B_{M+1} \in \calU'$. 
        But then $\Ball( D, {c}_e \ell / 3 )$ has diameter at most $c_e \ell$ and intersects $\Mult+1$ distinct members of $\calU'$,
        which is a contradiction. 
        Hence $\calU'$ has $({c}_e \ell/3)$-multiplicity $\Mult$.

        \item\ref{proposition:nagata:h}$\implies$\ref{proposition:nagata:g}:
        Clear, with $c_g \geq c_h'$.
        
        \item\ref{proposition:nagata:g}$\implies$\ref{proposition:nagata:h}:
        Assume that we have an open cover of $\calB = \calB_1 \cup \dots \cup \calB_\Mult$ that is a disjoint union 
        of $c_g \ell$-disjoint open families. We define 
        \begin{align*}
            \calB_k' := \left\{ \Ball(B,c_g \ell/3) \suchthat B \in \calB_k \right\}.
        \end{align*}
        Each $\calB_{k}'$ is $(c_g \ell/3)$-disjoint, 
        and $\calB' = \calB'_1 \cup \dots \cup \calB'_\Mult$ is an open cover with Lebesgue number $c_g \ell/3$. 
        Moreover, suppose $D \subseteq X$ has a diameter of at most $c_g \ell/3$.
        If $B, B' \in \calB'$ with $D \cap B \neq \emptyset$ and $D \cap B' \neq \emptyset$, then $B$ and $B'$ have distance at most $c_g \ell / 3$,
        and so $B \in \calB_k$ and $B' \in \calB_l$ with $k \neq l$.
		Hence $\calB'$ has $(c_g \ell/3)$-multiplicity $\Mult$.

        \item\ref{proposition:nagata:h}$\implies$\ref{proposition:nagata:f}: This is clear, with ${c}_f = c_h$ and ${c}_f' = c_h'$.
        
        \item\ref{proposition:nagata:f}$\implies$\ref{proposition:nagata:g}: 
Assume that~\ref{proposition:nagata:f} holds and let $\ell > 0$.
        There exists an $\ell$-bounded open cover $\calU = (U_{\alpha})_{\alpha\in\kappa}$ of $X$ 
        with Lebesgue number ${c}_f \ell$ and ${c}_f' \ell$-multiplicity $\Mult$.
        Without loss of generality, $\calU$ has at least $\Mult$ distinct members. 
        We let $\Lambda = (\lambda_{\alpha})_{\alpha\in\kappa}$ be a partition of unity subordinate to $\calU$. 
        By Theorem~\ref{theorem:mainresult}, we may assume that its members have Lipschitz constant $L = ( {c}_f \ell )^{-1} \max(1,\Mult-1)$.
        For every subset $\mu \subseteq \kappa$ with $\card{\mu} = k$ and $\epsilon, \zeta \in (0,1)$, we define \begin{align*}
            B_{\mu,\epsilon,\zeta}
            := 
            \left\{ 
                x \in X 
                \suchthat*
                \begin{array}{l}
                    \forall \alpha \in \mu                : \lambda_{\alpha}(x) > (\Mult-k+1) \epsilon, 
                    \\
                    \forall \beta  \in \kappa\setminus\mu : \lambda_{\beta }(x) < (\Mult-k+1) \zeta 
                \end{array}                
            \right\}.
        \end{align*}
        Consider any $\mu, \mu' \subseteq \kappa$ with $\card{\mu} = \card{\mu'} = k$ and $\mu \neq \mu'$.
        For all $\alpha \in \mu \setminus \mu'$, 
        \begin{gather*}
            \forall y  \in B_{\mu ,\epsilon,\zeta} : \lambda_{\alpha}(y ) > (\Mult-k+1) \epsilon,
            \qquad
            \forall y' \in B_{\mu',\epsilon,\zeta} : \lambda_{\alpha}(y') < (\Mult-k+1) \zeta.
        \end{gather*}
        Let us assume that $\epsilon > \zeta$. 
        For all $y \in B_{\mu ,\epsilon,\zeta}$ and $y' \in B_{\mu',\epsilon,\zeta}$,
        the Lipschitz property of $\lambda_{\alpha}$ now implies 
        \[
            \dist(y,y') 
            > 
            L^{-1} (\Mult-k+1) ( \epsilon - \zeta )
            = 
            \underbrace{\max(1,\Mult-1)^{-1} {c}_f (\Mult-k+1) ( \epsilon - \zeta )}_{:= 2 \cdot c_{g,k}}
            \cdot 
            \ell
        \]
        for $k = 1,\dots,\Mult-1$ and
        \[
            \dist(y,y') 
            > 
            L^{-1} \epsilon 
            = 
            \underbrace{\max(1,\Mult-1)^{-1} \epsilon {c}_f}_{:= 2 \cdot c_{g,\Mult}}
            \cdot 
            \ell
            .
        \]
        Consequently, still assuming $\epsilon > \zeta$, the families 
        \begin{align*}
            \calB_{k} := \left\{ B_{\mu,\epsilon,\zeta} \suchthat \mu \subseteq \kappa, \card{\mu}=k \right\}
        \end{align*}
        are $c_{g,k} \ell$-disjoint, with $c_{g,k}$ as above. 
        By construction, each $\calB_{k}$ is $\ell$-bounded, its members being subsets of members of $\calU$.
It remains to show that their union constitutes an open cover of $X$ for some choice of $\epsilon$ and $\zeta$.
        Let us assume that $\zeta < \epsilon$ and that 
        \begin{align*}
            (\Mult-k+1) \epsilon < (\Mult-k+2) \zeta, \quad k=1,\dots,\Mult.
        \end{align*}
        This is the case if and only if $\Mult(\Mult+1)^{-1} \epsilon < \zeta < \epsilon$.
Let $x \in X$. There exist pairwise distinct $\alpha_1,\dots,\alpha_\Mult \in \kappa$ such that $\lambda_{\alpha_1}(x) \geq \lambda_{\alpha_2}(x) \geq \dots \geq \lambda_{\alpha_\Mult}(x) \geq 0$.
We define the open sets 
        \begin{align*}
            B_{k,\epsilon,\zeta}
            := 
            \left\{ 
                x \in X 
                \suchthat* 
                \begin{array}{c}
                    \forall i \in \Indices{  1}{k} : \lambda_{\alpha_{i}}(x) > (\Mult-k+1) \epsilon, 
                    \\
                    \forall i \in \Indices{k+1}{\Mult} : \lambda_{\alpha_{i}}(x) < (\Mult-k+1) \zeta 
                \end{array}                
            \right\}.
        \end{align*}
        It suffices to show that $x \in B_{k,\epsilon,\zeta}$ for some $1 \leq k \leq \Mult$.
        We use a recursive argument for that.
        To begin with, if $\lambda_{\alpha_{\Mult}}(x) > \epsilon$, then $x \in B_{\Mult,\epsilon,\zeta}$.
        Otherwise, $\lambda_{\alpha_{\Mult}}(x) \leq \epsilon < \Mult^{-1}(\Mult+1) \zeta \leq 2 \zeta$. 
Going further, suppose that $k \in \Indices{2}{\Mult}$, that $x \notin B_{i,\epsilon,\zeta}$ for $i \in \Indices{k}{\Mult}$,
        and that 
        \begin{align*}
            \lambda_{\alpha_{\Mult  }}(x) < 2 \zeta,
            \quad 
            \lambda_{\alpha_{\Mult-1}}(x) < 3 \zeta,
            \quad 
            \dots,
            \quad 
            \lambda_{\alpha_{k      }}(x) < (\Mult-k+2) \zeta.
        \end{align*}
        If $\lambda_{\alpha_{k-1}}(x) > (\Mult-(k-1)+1)\epsilon$, then $x \in B_{k-1,\epsilon,\zeta}$.
        Otherwise, $\lambda_{\alpha_{k-1}}(x) \leq (\Mult-(k-1)+1) \epsilon < (\Mult-(k-1)+2) \zeta$ by assumption.
        We repeat this argument: if $x \notin B_{i,\epsilon,\zeta}$ for $i \in \Indices{1}{\Mult}$,
        then 
        \begin{align*}
            \lambda_{\alpha_{\Mult  }}(x) < 2 \zeta,
            \quad 
            \lambda_{\alpha_{\Mult-1}}(x) < 3 \zeta,
            \quad 
            \dots,
            \quad 
            \lambda_{\alpha_{1      }}(x) < (\Mult+1) \zeta.
        \end{align*}
        In that case,
        \begin{align*}
            \lambda_{\alpha_{1  }}(x) 
            + \dots + 
            \lambda_{\alpha_{\Mult  }}(x) 
            <
            \frac{(\Mult+1)(\Mult+2)-2}{2} \cdot \zeta.
        \end{align*}
        But the right-hand side is strictly less than $1$ provided that $\zeta$ is small enough.
        The latter is the case, for example, if $(\Mult+1)(\Mult+2) \epsilon < 2$. 
This contradicts the partition of unity property. 
        We conclude that $x \in B_{k,\epsilon,\zeta}$ for some $k \in \Indices{1}{\Mult}$.
        Lastly, possible specific choices of $\epsilon$ and $\zeta$ are 
        \begin{displaymath}
            \epsilon = 2(\Mult+2)^{-2},
            \quad
            \zeta
            =
            ( \Mult+0.5 )( \Mult+1 )^{-1}
            \epsilon
            .
        \end{displaymath}
        With that, $\epsilon - \zeta = 0.5 ( \Mult + 1 )^{-1} \epsilon = ( \Mult + 1 )^{-1} (\Mult+2)^{-2}$. This leads to  
        \suppress{
        \begin{align*}\color{blue}
            c_{g,\Mult} 
            \geq 
            2\max(1,\Mult-1)^{-1} (\Mult+2)^{-2} {c}_f
            , 
            \quad
            c_{g,k} 
&\geq\color{blue}
            \max(1,\Mult-1)^{-1} {c}_f (\Mult-k+1) ( \Mult + 1 )^{-1} (\Mult+2)^{-2}
            \\&\geq\color{blue}
            \max(1,\Mult-1)^{-1} {c}_f 2 ( \Mult + 1 )^{-1} (\Mult+2)^{-2}
            \\&\geq\color{blue}
            2\max(1,\Mult-1)^{-1} ( \Mult + 1 )^{-1} (\Mult+2)^{-2} {c}_f
        \end{align*}
        }\begin{align*}
            \suppress
            {\color{blue}c_{g,k} \geq c_{g} := \frac{{c}_f}{\Mult(\Mult+2)^2\max(1,\Mult-1)}, \quad k = 1,\dots,\Mult\\}
            {c_{g,k} \geq c_{g} := \max(1,\Mult-1)^{-1} ( \Mult + 1 )^{-1} (\Mult+2)^{-2} {c}_f, \quad k = 1,\dots,\Mult}
.
        \end{align*}
        This shows the desired implication.

        \item\ref{proposition:nagata:a}$\implies$\ref{proposition:nagata:d}:
        Let $\calV$ be an open cover by all balls of radius $\ell/2$.
        Then $\calV$ has Lebesgue number $\ell/2$. 
        Assuming that~\ref{proposition:nagata:a} is true,
        we extract an open refinement with Lebesgue number $c_a \ell/2$ and multiplicity $\Mult$. 
        Its members have diameter at most $\ell$, and so~\ref{proposition:nagata:d} holds with ${c}_d \geq c_a/2$. 

        \item\ref{proposition:nagata:d}$\implies$\ref{proposition:nagata:a}:
        Let $\calU = (U_{\alpha})_{\alpha\in\kappa}$ is any open cover of $X$ with Lebesgue number $2\Leb > 0$.
        Assuming that~\ref{proposition:nagata:d} holds,
        there exists an $\Leb$-bounded open cover $\calV$ of $X$ which has multiplicity $\Mult$ and Lebesgue number ${c}_d \Leb$.
        For each $V \in \calV$, we fix $\alpha(V) \in \kappa$ such that $V \subseteq U_{\alpha}$.
        We set
        \begin{align*}
            W_{\alpha} := \bigcup\left\{ V \in \calV \suchthat \alpha(V) = \alpha \right\}, \quad \alpha \in \kappa.
        \end{align*}
        Then $\calW := (W_{\alpha})_{\alpha\in\kappa}$ is an open refinement of $\calU$.
        One easily sees that $\calW$ has multiplicity $\Mult$.
        The ball around any $x \in X$ of radius ${c}_d \Leb$ is contained in some member of $\calV$,
        which must be contained in some member of $\calW$.
        Hence~\ref{proposition:nagata:a} holds with $c_a \geq {c}_d$.

        \item\ref{proposition:nagata:h}$\implies$\ref{proposition:nagata:i}:
        Let $\calU$ be an open cover of $X$ with Lebesgue number $2\Leb$.
        If~\ref{proposition:nagata:h} is true,
        then we find pairwise disjoint $\Leb$-bounded $c_h' \Leb$-disjoint open families $\calB_1, \dots, \calB_\Mult$ 
        such that $\calB = \calB_1 \cup \dots \cup \calB_\Mult$ is a cover of $X$ with Lebesgue number $c_h \Leb$ and $c_h' \Leb$-multiplicity $\Mult$. 
        Then every member of $\calB$ is contained in a member of $\calU$, so $\calB$ refines $\calU$. 

        \item\ref{proposition:nagata:i}$\implies$\ref{proposition:nagata:h}:
        Let $\ell > 0$ and $\calU$ be the family of all balls in $X$ of radius $\ell$. 
        Then $\calU$ has Lebesgue number $\ell/2$. If~\ref{proposition:nagata:i} holds, 
        then we find pairwise disjoint $\ell/2$-bounded $c_i' \ell/2$-disjoint open families $\calB_1, \dots, \calB_\Mult$ 
        such that $\calB = \calB_1 \cup \dots \cup \calB_\Mult$ is a cover of $X$ 
        with Lebesgue number $c_i \ell/2$.
        It follows that $\calB$ has $c_i'/2$-multiplicity $\Mult$. 
        
    \end{itemize}
    These implications combined complete the proof. 
\end{proof}

\begin{corollary}
    A metric space $X$ has Assouad--Nagata dimension at most $n$
    if and only if 
    there exists $c \in (0,1)$ such that every open cover $\calU = (U_{\alpha})_{\alpha\in\kappa}$ of $X$ with Lebesgue number $L$ 
    has an open refinement $\calW = (W_{\alpha})_{\alpha\in\kappa}$ with Lebesgue number $c L$ and multiplicity $n+1$. \qed
\end{corollary}
\suppress{
\begin{proof}\color{red}
    If item~\ref{proposition:nagata:e} holds, then for every $\ell > 0$ we have an $\ell$-bounded open cover 
    such that every ball of radius $c_e \ell / 2$ intersects with at most $n+1$ members of the cover. 
    In turn, if there exists $c > 0$ such that for every $\ell > 0$ there exists an $\ell$-bounded open cover 
    such that every ball of radius $c \ell$ intersects with at most $n+1$ members of the cover, 
    then every set of diameter at most $c \ell$ intersects with at most $n+1$ members of the cover. 
\end{proof}
}

\begin{remark}
    The numerous conditions in Proposition~\ref{proposition:nagata},
    which are all equivalent to having Assouad--Nagata dimension at most $n = \Mult - 1$, 
    are inspired by analogous conditions for dimensions of metric spaces,
    such as the hyperbolic, uniform, asymptotic, or asymptotic Assouad--Nagata dimensions~\cite{assouad1982sur,brodskiy2005coarse,lang2005nagata,dadarlat2007uniform,dranishnikov2007asymptotic,buyalo2008hyperbolic,brodskiy2009assouad,dranishnikov2009cohomological,repovvs2010asymptotic,cencelj2013asymptotic,le2015assouad,dydak2016large,xia2019strong}.
    We have included some characterizations of the Assouad--Nagata dimension in~\cite[Propositon~2.5]{lang2005nagata} and~\cite[Proposition~2.2]{brodskiy2009assouad}. 
    We also point out that the~\ref{proposition:nagata:f}$\implies$\ref{proposition:nagata:g} part of the preceding proof resembles the proof of Lemma~3 in~\cite{brodskiy2005coarse}, where asymptotic and coarse dimensions are addressed, 
    and circumvents the explicit use of the nerve complex mapping as in~\cite[Proposition~2.5]{lang2005nagata}.
    Moreover,
    formulation~\ref{proposition:nagata:i} resembles a result that is originally due to Stone, which states that every open cover of a paracompact space has a $\sigma$-discrete locally finite refinement;~see~\cite[Lemma~5.1.16]{engelking1989general} and   also~\cite{burke1984covering}.

    Assouad~\cite{assouad1982sur} introduced the Assouad--Nagata dimension of a metric space 
    under the name Nagata dimension, in his reception of Nagata's work~\cite{nagata1960metric}. 
    Important results related to large-scale geometry were proven by Lang and Schlichenmaier~\cite{lang2005nagata} more than 20 years later.
The Assouad--Nagata dimension is at least the covering dimension~\cite[Section~5]{le2015assouad}. 
    Metrizable topological spaces of covering dimension at most $n$ 
    can be metrized to have Assouad--Nagata dimension at most $n$~\cite[(1.7)]{assouad1982sur}.

    Proposition~\ref{proposition:nagata} implies that any metric space $X$ with Assouad--Nagata dimension at most $n$ must also have asymptotic and uniform dimension at most $n$.
Here, the \emph{uniform dimension} and \emph{asymptotic dimension} are defined as follows~\cite[Section~3]{dydak2016large}.
    A metric space $X$ has \emph{uniform dimension} at most $n$
    if for all $L > 0$ there exists $l > 0$ such that 
    every open cover of $X$ with Lebesgue number $L$
    has an open refinement with Lebesgue number $l$ and multiplicity $n+1$.
    Similarly, $X$ has \emph{asymptotic dimension} at most $n$
    if for all $l > 0$ there exists $L > 0$ such that 
    every open cover of $X$ with Lebesgue number $L$
    has an open refinement with Lebesgue number $l$ and multiplicity $n+1$. 
\end{remark}

We characterize the Assouad--Nagata dimension via Lipschitz maps into a polyhedron. 
While this is analogous to one of the known definitions of the asymptotic dimension, 
the Assouad--Nagata again comes with a linear quantitative relationship.
A few preparations are in order. 
We have already established in Section~\ref{section:topology} that the vectorization of a pointwise finite partition of unity subordinate to a cover $\calU = (U_\alpha)_{\alpha\in\kappa}$
takes values in the vector space $\bbR^{(\kappa)}$. 
More specifically, the vectorization maps into a simplicial complex $\calK \subseteq \bbR^{(\kappa)}$ whose simplices have dimension at most $n$ and whose vertices are canonical unit vectors of $\bbR^{(\kappa)}$.
Consequently, the vectorization can be seen as mapping $X$ into $\calK$.

\begin{proposition}[{compare~\cite[Assertion~1]{bell2004asymptotic}}] 
    Let $X$ be a metric space. 
    Then $X$ has Assouad--Nagata dimension at most $n$ 
    if and only if 
    there exist a constant $c_\star > 0$ such that 
    for every $\ell > 0$ there exist a set $\kappa$ and $\Lambda : X \to \calK \subseteq \bbR^{(\kappa)}_{\ell^{2}}$ with Lipschitz constant $c_\star/\ell$
    into a simplicial complex $\calK$ of dimension at most $n$
    whose vertices are canonical unit vectors of $\bbR^{(\kappa)}$
    such that the preimage of every simplex of $\calK$ under $\Lambda$ has diameter at most $\ell$. 
    
    Letting $c_{d} > 0$ denote the constant in Proposition~\ref{proposition:nagata}, we may assume
    \begin{align*}
        \frac{1}{2}
		\left( n(n+1)\right)^{-\frac{1}{2}} 
\cdot c_{d}^{-1} 
        \leq 
        c_{\star} 
        \leq 
        \sqrt{2} n \cdot c_{d}^{-1}.
\end{align*}
\end{proposition}
\begin{proof}
    First, suppose that $X$ has Assouad--Nagata dimension at most $n$.
    So there exists $c > 0$ such that for every $\ell > 0$ 
    there exists an $\ell$-bounded cover $\calU = (U_\alpha)_{\alpha\in\kappa}$ of $X$ with multiplicity $n + 1$ and with Lebesgue number $c \ell$.
    By Corollary~\ref{corollary:vectorizationestimate}, there exists a strongly subordinate partition of unity 
    whose vectorization $\Lambda : X \to \bbR^{(\kappa)}_{\ell^{2}}$ has Lipschitz constant $\sqrt{2} n / (c \ell)$. 
We define a simplicial complex $\calK \subseteq \bbR^{(\kappa)}$ by the following selection criterion:
    if $\mu \subseteq \kappa$ such that there exists $x \in X$ with $\lambda_\alpha(x) > 0$ for all $\alpha \in \mu$,
    then $\calK$ contains the simplex, called $K_{\mu}$, of dimension $\card{\mu}-1$
    whose vertices are unit vectors $(e_{\alpha})_{\alpha\in\mu}$.
    Obviously, $\calK$ is a simplicial complex containing $\Lambda(X)$.
Given $\mu \subseteq \kappa$ and $K_{\mu} \in \calK$,
    we observe $\Lambda^{-1}(K_{\mu}) \subseteq \bigcap_{\alpha\in\mu} U_{\alpha}$.
    The former intersection has a diameter of at most $\ell$. 
    This shows one direction of the statement with $c_\star \leq \sqrt{2} n / c$.

    Second, suppose that 
    there exist $c_\star > 0$ such that 
    for every $\ell > 0$ there exist a set $\kappa$ and a $c_\star \ell^{-1}$-Lipschitz map $\Lambda : X \to \calK \subseteq \bbR^{(\kappa)}_{\ell^{2}}$
    into a simplicial complex $\calK$ of dimension at most $n$
    whose vertices are canonical unit vectors of $\bbR^{(\kappa)}$
    such that the preimage of every simplex of $\calK$ under $\Lambda$ has diameter at most $\ell$. 
Fixing such $\ell > 0$ and $\Lambda$,
	we let $\calV = (V_\alpha)_{\alpha\in\kappa}$ be the cover of $\calK$ by the open stars around the vertices. 
    $\calV$ has multiplicity $n+1$ and a Lebesgue number $\Leb_{n} \geq H/(n+1)$ since the simplices have height $H := \sqrt{n+1} / \sqrt{n}$ at every vertex. 
    We define an open cover $\calU = (U_\alpha)_{\alpha\in\kappa}$ of $X$ by setting $U_\alpha := \Lambda^{-1}(V_\alpha)$. 
Obviously, $\calU$ has multiplicity $n+1$ and is $2\ell$-bounded. Moreover, $\Lambda$ maps every open ball of radius $\Leb_n c_\star^{-1} \ell$ into an open ball of radius $\Leb_n$,
    which is included in a member of $\calV$. Hence, $\calU$ has Lebesgue number $\Leb_n c_\star^{-1} \ell$. 
    In summary, $X$ has Assouad--Nagata dimension at most $n$, and the other inequality follows. 
\end{proof}

To round up the discussion with a concrete example, 
we show that subsets of Euclidean space have finite Assouad--Nagata dimensions, 
and we provide explicit estimates for the relevant constants.

\begin{lemma}
    Let $X \subseteq \bbR^{n}$. 
    Every open cover $\calU = (U_{\alpha})_{\alpha\in\kappa}$ of $X$ with Lebesgue number $\Leb$
    has an open refinement $\calW = (W_{\alpha})_{\alpha\in\kappa}$ with Lebesgue number $((n+1)\sqrt{8n})^{-1} \Leb$ and multiplicity $n+1$,
    and $\calW$ has a subordinate partition of unity whose partial sums have Lipschitz constant $\sqrt{8n}/\Leb$.
\end{lemma}

\begin{proof}
    To begin with, let $\ell > 0$.
    We write $\bbT_{\ell}$ for the Freudenthal--Kuhn triangulation of $\bbR^{n}$ scaled to length $\ell/2$.
    This is a triangulation of $\bbR^{n}$ which consists of $n$-dimensional simplices of diameter at most $D := \ell/2$ and with minimum height at least $H := \ell / \sqrt{8n}$. The set of vertices of this triangulation is $\bbV_{\ell} := (\ell/2) \bbZ^{n}$.
    We refer to the literature~\cite{freudenthal1942simplizialzerlegungen,bey2000simplicial,kachanovich2019maillage} for more details.
For every vertex $z \in \bbV_{\ell}$ of the scaled Freudenthal--Kuhn triangulation, 
    $\lambda_z : \bbR^{n} \rightarrow \bbR$ denotes the non-negative function that is affine over each simplex of $\bbT_{\ell}$,
    that satisfies $\lambda_z(z) = 1$, and that vanishes at all other vertices of $\bbT_{\ell}$.
    Moreover, we let $C_z$ be the open star of $z \in \bbV_{\ell}$ within the triangulation,
    which coincides with the set of points over which $\lambda_z$ is positive. 
    The set $C_z$ is convex and open, contains the open ball around $z$ of radius $H$, and its diameter is at most $2D$.

    Suppose that $x \in \bbR^{n}$ with $\lambda_z(x) = \lambda \in (0,1]$.
    Then $x \in C_z$. 
    The ray starting from $z$ and going through $x$ hits the boundary of $C_{z}$ at some point $x' \in \partial C_z$.
Then $x = z + (1-\lambda)( x' - z )$. 
    Since $C_z$ is convex, 
    \begin{align*}
        \lambda( C_z - z ) + x
        &=
        \lambda( C_z - z ) + z + (1-\lambda)( x' - z )
=
        \lambda C_z  + (1-\lambda) x' 
        .
    \end{align*}
    Since $\lambda > 0$, we conclude 
    $\lambda( C_z - z ) + x \subseteq C_z$.
    But then 
$\Ball\left( x, \lambda_z(x) H \right) \subseteq C_z$.
Since the family $(\lambda_z)_{z \in \bbV_{\ell}}$ is a partition of unity,
    and every point is contained within the support of at most $n+1$ of these partition of unity functions, 
    for every $x \in \bbR^n$ we can find $z \in \bbV_{\ell}$ such that $\lambda_z(x) \geq \frac{1}{n+1}$.
    In particular, for every $x \in \bbR^n$ there exists $z \in \bbV_{\ell}$ such that 
$\Ball\left( x, (n+1)^{-1}H \right) \subseteq C_z$.
In summary, 
$\calV_{\ell} := \left\{ C_z \suchthat* z \in \bbV_{\ell} \right\}$
is an $\ell$-bounded open cover of $\bbR^n$ with Lebesgue number $H / (n+1)$.
    For each $x \in X$ there at most $n+1$ distinct $z \in \bbV_{\ell}$ such that $x \in C_z$, and so $\calV_{\ell}$ has multiplicity $n+1$. 
    
    Suppose that $\calU$ is an open cover of $X$ with Lebesgue number $\Leb$ and consider the case $\ell = \Leb$. 
    When $z \in \bbV_{\Leb}$ with $C_z \cap X \neq \emptyset$,
    then we can fix $\alpha(z) \in \kappa$ with $C_{z} \cap X \subseteq U_{\alpha}$. 
    For each $\alpha \in \kappa$, we define 
\begin{align*}
        W_{\alpha}       := \bigcup\left\{ C_z \cap X \suchthat* z \in \bbV_{\Leb}, C_z \cap X \neq \emptyset, \alpha(z) = \alpha \right\},
        \qquad 
        \lambda_{\alpha} := \inlinesum_{ z \in \bbV_{\Leb}, C_z \cap X \neq \emptyset, \alpha(z) = \alpha } \lambda_{z}.
    \end{align*}
    For any $x \in X$, we have $\Ball(x,H/(n+1)) \subseteq C_z$ for some $z \in \bbV_{\Leb}$ that satisfies $C_z \cap X \neq \emptyset$,
    and hence $\Ball(x,H/(n+1)) \subseteq W_{\alpha(z)}$. 
    So $\calW = (W_{\alpha})_{\alpha\in\kappa}$ is an open refinement of $\calU$ with multiplicity $n+1$ with Lebesgue number $H/(n+1)$. 
    The family $\Lambda := (\lambda_{\alpha})_{\alpha\in\kappa}$ is a locally finite partition of unity subordinate to $\calW$.
    Any partial sum $\lambda_{Z} := \inlinesum_{z \in Z} \lambda_{z}$ with $Z \subseteq \bbV_{\Leb}$ has Lipschitz constant $H^{-1} = \sqrt{8n} / \Leb$.
\end{proof}

\begin{remark}
The partition of unity found in the proof above consists of functions with Lipschitz constant $\sqrt{8n} / \Leb$,
    exploiting the special situation of Euclidean space.
    Notably, this is better than the general upper bound $(n+1)n\sqrt{8n} / \Leb$ guaranteed by our main result.  
\end{remark}


\begin{thebibliography}{10}

\bibitem{assouad1982sur}
Patrice Assouad.
\newblock Sur la distance de {N}agata.
\newblock {\em C. R. Acad. Sci. Paris S{\'e}r. I Math.}, 294(1):31--34, 1982.

\bibitem{austin2014partitions}
Kyle Austin and Jerzy Dydak.
\newblock Partitions of unity and coverings.
\newblock {\em Topology and its Applications}, 173:74--82, 2014.

\bibitem{aviles2019complete}
Antonio Avil{\'e}s and Gonzalo Mart{\'\i}nez-Cervantes.
\newblock Complete metric spaces with property {$(Z)$} are length spaces.
\newblock {\em Journal of Mathematical Analysis and Applications},
  473(1):334--344, 2019.

\bibitem{bell2004asymptotic}
G.~Bell and Alexander Dranishnikov.
\newblock On asymptotic dimension of groups acting on trees.
\newblock {\em Geometriae Dedicata}, 103:89--101, 2004.

\bibitem{bell2008asymptotic}
Gregory Bell and Alexander Dranishnikov.
\newblock Asymptotic dimension.
\newblock {\em Topology and its Applications}, 155(12):1265--1296, 2008.

\bibitem{bell2003property}
Gregory~C. Bell.
\newblock Property {A} for groups acting on metric spaces.
\newblock {\em Topology and its Applications}, 130(3):239--251, 2003.

\bibitem{bey2000simplicial}
J{\"u}rgen Bey.
\newblock Simplicial grid refinement: On {F}reudenthal's algorithm and the
  optimal number of congruence classes.
\newblock {\em Numerische Mathematik}, 85(1):1--29, 2000.

\bibitem{brodskiy2005coarse}
Nikolay Brodskiy and Jerzy Dydak.
\newblock Coarse dimensions and partitions of unity.
\newblock {\em Rev. R. Acad. Cien. Serie A. Mat.}, 102(1):1--19, 2008.

\bibitem{brodskiy2009assouad}
Nikolay Brodskiy, Jerzy Dydak, Josem Higes, and Atish Mitra.
\newblock {Assouad}--{Nagata} dimension via {Lipschitz} extensions.
\newblock {\em Israel Journal of Mathematics}, 171(1):405--423, 2009.

\bibitem{burke1984covering}
Dennis~K. Burke.
\newblock Covering properties.
\newblock In {\em Handbook of Set-Theoretic Topology}, pages 347--422.
  Elsevier, 1984.

\bibitem{buyalo2008hyperbolic}
Sergei Buyalo and Viktor Schroeder.
\newblock Hyperbolic dimension of metric spaces.
\newblock {\em St. Petersburg Mathematical Journal}, 19(1):67--76, 2008.

\bibitem{cencelj2013asymptotic}
Matija Cencelj, Jerzy Dydak, and Ale{\v{s}} Vavpeti{\v{c}}.
\newblock Asymptotic dimension, property {A}, and {L}ipschitz maps.
\newblock {\em Revista matem{\'a}tica complutense}, 26:561--571, 2013.

\bibitem{cobzacs2019lipschitz}
{\c{S}}tefan Cobza{\c{s}}, Radu Miculescu, and Adriana Nicolae.
\newblock {\em Lipschitz Functions}.
\newblock Springer, 2019.

\bibitem{dadarlat2007uniform}
Marius Dadarlat and Erik Guentner.
\newblock Uniform embeddability of relatively hyperbolic groups.
\newblock {\em Journal f\"ur die reine und angewandte Mathematik}, 612:1--15,
  2007.

\bibitem{dranishnikov2009cohomological}
Alexander Dranishnikov.
\newblock Cohomological approach to asymptotic dimension.
\newblock {\em Geometriae Dedicata}, 141:59--86, 2009.

\bibitem{dranishnikov2007asymptotic}
Alexander~N. Dranishnikov and Justin Smith.
\newblock On asymptotic {Assouad--Nagata} dimension.
\newblock {\em Topology and its Applications}, 154(4):934--952, 2007.

\bibitem{duan2010property}
Yujuan Duan, Qin Wang, and Xianjin Wang.
\newblock Property {A} and uniform embeddability of metric spaces under
  decompositions of finite depth.
\newblock {\em Chinese Annals of Mathematics, Series B}, 31(1):21--34, 2010.

\bibitem{dydak2003partitions}
Jerzy Dydak.
\newblock Partitions of unity.
\newblock {\em Topology Proceedings}, 27:125--171, 2003.

\bibitem{dydak2016large}
Jerzy Dydak and Atish~J. Mitra.
\newblock Large scale absolute extensors.
\newblock {\em Topology and its Applications}, 214:51--65, 2016.

\bibitem{engelking1989general}
Ryszard Engelking.
\newblock {\em General Topology}, volume~6.
\newblock Heldermann Verlag, 1989.

\bibitem{freudenthal1942simplizialzerlegungen}
Hans Freudenthal.
\newblock {Simplizialzerlegungen von beschr\"ankter Flachheit}.
\newblock {\em Annals of Mathematics}, pages 580--582, 1942.
\newblock English translation by Mathijs Wintraecken available at:
  https://arxiv.org/abs/2302.11922.

\bibitem{fried1984open}
Jan Fried.
\newblock Open cover of a metric space admits $\ell^\infty$-partition of unity.
\newblock {\em Proceedings of the 11th Winter School on Abstract Analysis},
  pages 139--140, 1984.

\bibitem{frolik1984existence}
Zden\v{e}k Frol\'ik.
\newblock Existence of $\ell^\infty$-partitions of unity.
\newblock {\em Rend. Sem. Mat. Univers. Politecn. Torino}, 42(1), 1984.

\bibitem{garling2007inequalities}
David J.~H. Garling.
\newblock {\em Inequalities: a journey into linear analysis}.
\newblock Cambridge University Press, 2007.

\bibitem{kachanovich2019maillage}
Siargey Kachanovich.
\newblock {\em Maillage de vari{\'e}t{\'e}s avec les triangulations de
  {C}oxeter}.
\newblock PhD thesis, Universit{\'e} C{\^o}te d'Azur (ComUE), 2019.

\bibitem{lang2005nagata}
Urs Lang and Thilo Schlichenmaier.
\newblock {Nagata} dimension, quasisymmetric embeddings, and {Lipschitz}
  extensions.
\newblock {\em International Mathematics Research Notices},
  2005(58):3625--3655, 2005.

\bibitem{le2015assouad}
Enrico Le~Donne and Tapio Rajala.
\newblock {Assouad} dimension, {Nagata} dimension, and uniformly close metric
  tangents.
\newblock {\em Indiana University Mathematics Journal}, pages 21--54, 2015.

\bibitem{luukkainen1977elements}
Jouni Luukkainen and Jussi V\"ais\"al\"a.
\newblock Elements of {Lipschitz} topology.
\newblock {\em Annales Academiae Scientiarum Finnicae}, 3:85--122, 1977.

\bibitem{nagata1960metric}
Jun-iti Nagata.
\newblock On a metric characterizing dimension.
\newblock {\em Proceedings of the Japan Academy}, 36(6):327--331, 1960.

\bibitem{repovvs2010asymptotic}
Du{\v{s}}an Repov{\v{s}} and Mykhailo Zarichnyi.
\newblock On the asymptotic extension dimension.
\newblock {\em Ukrains'kyi Matematychnyi Zhurnal}, 62(11):1523--1530, 2010.

\bibitem{robinson1983uniformly}
Stewart~M. Robinson and Zachary Robinson.
\newblock Uniformly continuous partitions of unity on a metric space.
\newblock {\em Canadian Mathematical Bulletin}, 26(1):115--117, 1983.

\bibitem{xia2019strong}
Jun Xia and Xianjin Wang.
\newblock Strong embeddability for groups acting on metric spaces.
\newblock {\em Chinese Annals of Mathematics, Series B}, 40(2):199--212, 2019.

\end{thebibliography}
\end{document}